\documentclass[reqno]{amsart}

\usepackage{amsfonts, amscd, epsfig, amsmath, amssymb,enumerate}
\usepackage{amstext}
\usepackage{graphicx}
\usepackage[usenames,dvipsnames,svgnames,table]{xcolor}
\usepackage{mathrsfs}
\usepackage{mathptmx}       
\usepackage{helvet}         
\usepackage{courier}        
\usepackage{todonotes}
\usepackage[charter]{mathdesign}
\usepackage{cancel}

\usepackage{a4wide}
\definecolor{brightcerulean}{rgb}{0.11, 0.67, 0.84}
\definecolor{cerulean}{rgb}{0.0, 0.48, 0.65}
\definecolor{Gray}{rgb}{0.5, 0.5, 0.5}
\definecolor{brightcerulean}{rgb}{0.11, 0.67, 0.84}
\definecolor{cerulean}{rgb}{0.0, 0.48, 0.65}
\definecolor{Gray}{rgb}{0.5, 0.5, 0.5}

\definecolor{columbiablue}{rgb}{0.61, 0.87, 1.0}
\usepackage{hyperref}
\usepackage{fancyhdr}

\definecolor{aleacolor}{rgb}{0.16,0.59,0.78}

\hypersetup{
breaklinks,
colorlinks=true,
linkcolor=aleacolor,
urlcolor=aleacolor,
citecolor=aleacolor}

\theoremstyle{plain}

\makeatletter \@addtoreset{equation}{section} \makeatother

\newtheorem{thm}{Theorem}[section]
\newtheorem{lem}[thm]{Lemma}
\newtheorem{cor}[thm]{Corollary}

\newtheorem{prop}[thm]{Proposition}

\newtheorem{rem}[thm]{Remark}

\newcommand\bE{{\mathbb E}}
\newcommand\bN{{\mathbb N}}
\newcommand\bP{{\mathbb P}}
\newcommand\bQ{{\mathbb Q}}
\newcommand\bR{{\mathbb R}}
\newcommand\bT{{\mathbb T}}

\newcommand\PP{{\mathbb P}}

\newcommand\R{{\mathbb R}}

\newcommand\Z{{\mathbb Z}}

\newcommand{\mc}[1]{{\mathcal #1}}

\newcommand{\bb}[1]{{\mathbb #1}}

\newcommand{\gab}[1]{\textcolor{blue}{#1}}

\renewcommand{\bar}{\overline}
\renewcommand{\leq}{\leqslant}

\renewcommand{\geq}{\geqslant}

\newcommand{\cro}[1]{\left[#1\right]}
\newcommand{\bra}[1]{\left\{#1\right\}}
\newcommand{\pa}[1]{\left(#1\right)}
\newcommand{\abs}[1]{\left|#1\right|}

\let\oldtocsection=\tocsection
\let\oldtocsubsection=\tocsubsection
\let\oldtocsubsubsection=\tocsubsubsection
\renewcommand{\tocsection}[2]{\hspace{0em}\oldtocsection{#1}{#2}}
\renewcommand{\tocsubsection}[2]{\hspace{1em}\oldtocsubsection{#1}{#2}}
\renewcommand{\tocsubsubsection}[2]{\hspace{2em}\oldtocsubsubsection{#1}{#2}}
\DeclareRobustCommand{\SkipTocEntry}[5]{}

\title[]{Hydrodynamics for  SSEP with non-reversible\\ slow  boundary dynamics: Part II, below the critical regime}
\author{C. Erignoux}
\address{Cl\'ement Erignoux, Equipe PARADYSE, Bureau B211
Centre INRIA Lille Nord-Europe
Park Plaza, Parc scientifique de la Haute-Borne, 40 Avenue Halley B\^atiment B, 59650 Villeneuve-d'Ascq
France}
\email{{\tt clement.erignoux@inria.fr}}

\author{P.  Gon\c calves}
\address{Patr\'icia Gon\c calves, Center for Mathematical Analysis,  Geometry and Dynamical Systems,
Instituto Superior T\'ecnico, Universidade de Lisboa,
Av. Rovisco Pais, 1049-001 Lisboa, Portugal.}
\email{{\tt pgoncalves@tecnico.ulisboa.pt}}

\author{G. Nahum}
\address{Gabriel Nahum, Center for Mathematical Analysis,  Geometry and Dynamical Systems, Instituto Superior T\'ecnico, Universidade de Lisboa,
	Av. Rovisco Pais, 1049-001 Lisboa, Portugal.}
\email{{\tt gabriel.nahum@tecnico.ulisboa.pt}}

\thanks{C.E. gratefully acknowledges funding from the European Research Council under the European Unions Horizon 2020 Program, ERC Consolidator GrantUniCoSM (grant agreement no 724939). P.G. thanks FCT/Portugal for support through the project UID/MAT/04459/2013. G.N thanks FCT/Portugal for the support through the project Lisbon Mathematics PhD (LisMath). This project has received funding from the European Research Council (ERC) under  the European Union's Horizon 2020 research and innovative programme (grant agreement   No 715734).}


\date{\today.}

\begin{document}

\maketitle

\begin{abstract}
The purpose of this article is to provide a simple proof of the hydrodynamic and hydrostatic behavior of the SSEP in contact with reservoirs which inject and remove particles in a finite size windows at the extremities of the bulk. More precisely, the reservoirs inject/remove particles at/from any point of a window of size $K$ placed at each extremity of the bulk and particles are injected/removed to the first open/occupied position in that window. The reservoirs have slow dynamics, in the sense that they intervene at speed $N^{-\theta}$ w.r.t. the bulk dynamics. In the first part of this article,  we treated the case $\theta>1$ for which the entropy method can be adapted. We treat here the case where the boundary dynamics is too fast for the Entropy Method to apply. We prove using duality estimates inspired by previous work that the hydrodynamic limit is given by the heat equation with Dirichlet boundary conditions, where the density at the boundaries is fixed by the parameters of the model.
\end{abstract}

\section{introduction}

We consider a $1$-dimensional boundary driven lattice gas, whose boundary dynamics is non-reversible with respect to product measures, and which is slowed down by an extra factor $N^{-\theta}$ with respect to the diffusive scaling $N^2$ of the bulk SSEP dynamics. In the first part of this article \cite{EGN}, the constant $\theta$ is assumed to be larger or equal to $1$, which allowed us to adapt the classical entropy method to our non-reversible dynamics, to derive both the hydrodynamic and hydrostatic limits under suitable technical assumptions. In the case $\theta=1$, we show that both scaling limits exhibit so-called non-linear Robin boundary conditions, whereas in the $\theta>1$ regime, the boundary dynamics is to slow to be visible in the diffusive time scale and the hydrodynamic limit is ruled by the heat equation with Neumann boundary conditions. 

In the case $0<\theta <1$, however, the non-reversible boundary dynamics generates entropy w.r.t. equilibrium product measure at a fast rate, so that the classical entropy estimates are not sharp enough to derive the hydrodynamics nor the hydrostatic limit. For this reason, we instead adapt the tools exploited in \cite{ELX18,E18} (non-reversible dynamics, diffusive scaling of the boundary dynamics) to our non-reversible slowed down case, to prove that the macroscopic behavior of the  system started close to a given profile $f_0$ is ruled by the heat equation with Dirichlet boundary conditions
\begin{equation*}
\label{hydroInt}
 \begin{cases}
 \partial_{t}\rho_{t}(u)=  \Delta\, {\rho} _{t}(u)& \mbox{for }  t\geq 0,\; u\in(0,1)\\
 \rho _{t}(0)= \alpha  &\mbox{for }   t> 0\\
 \rho_t(1)= \alpha' &  \mbox{for }  t> 0\\
 \rho_0(\cdot)= f_0(\cdot)&
 \end{cases},
\end{equation*}
and that the macroscopic stationary profile is given by the linear interpolation $\rho^*(u)=\alpha+u(\alpha'-\alpha)$.
The general strategy, both for the hydrostatics and the hydrodynamics described above, is to directly estimate the discrete density field $\rho^N$ defined in \eqref{eq:densityfield} and the 2-point correlation field $\varphi^N$ defined in \eqref{eq:Defphi}, which both solve discrete difference equations. Due to the non-reversible dynamics, however, these equations are not closed. In \cite{ELX18,E18} (in the regime  $\theta=0$), this difficulty is solved by artificially closing  the equations with unknown boundary terms, and then estimating the missing terms by  using duality. In \cite{DMP12}, De Masi et al. treated a non-reversible boundary dynamics in the slow scaling, corresponding here to the case $\theta=1$, and instead completely estimate the cascading $n$-points correlation field. They prove that for $n$ large enough, those correlation field vanish, and this is sufficient to backtrack those estimates to the initial $2$-points correlation field.  This involves significant technical and phenomenological difficulties which can be overcome in the case $\theta\geq 1$ without so much effort. In fact,  their proof can likely be extended to the case $\theta>1/2$ with a bit more work, but we do not pursue this issue here since our proof can cover all the cases $\theta\in(0,1)$.

In the present article, we bridge the gap between \cite{ELX18,E18} (case $\theta=0$), and \cite{DMP12} (case $\theta=1$), and consider the case $ 0<\theta<1$.
In order to avoid to try and obtain the same type of delicate $n$-points correlation estimates used in \cite{DMP12}, we adapt the strategy of \cite{ELX18,E18}, and under the assumption \eqref{ass:thetal1} below, we are able to artificially close the discrete difference equations for the density and for the $2$-points correlation field. More precisely, we consider a boundary dynamics with two distinct mechanisms, one representing the system being put in contact with a infinite equilibrium reservoir (occurring at rate $r$), and a second non-reversible creation/annihilation process whose rates $b(\eta)$ depend on the local configuration at the boundary. The key point of the proof is the study of a branching process, representing the interaction of a site whose value is currently unknown with the rest of the system. Whenever a reservoir is queried (rate $r$), one of its branches dies, whereas when the non-reversible event occurs, the process branches out.  Assumption \eqref{ass:thetal1} allows us, roughly speaking, to close the equation, by ensuring that this dual branching process eventually dies out, thus allowing us to determine the value of the unknown site. Note that the techniques developed here are not specific to the choice of the function $b$ made in this article, and can be applied to any local non-reversible dynamics  under an assumption analogous to \eqref{ass:thetal1}.

Because of the slow boundary, unlike in \cite{ELX18,E18} the dual branching process determining the value of the density filed at the boundary self-decorrelates. As a consequence, one obtains explicit formulae (cf. \eqref{eq:alpha}) for the macroscopic Dirichlet boundary conditions $\alpha$ and $\alpha'$ appearing in the hydrodynamic limit \eqref{hydroInt}.
Because of this self-decorrelation property, one could expect that, unlike in the case $\theta=0$, assumption \eqref{ass:thetal1} could possibly be dismissed. This, however, would involve fundamental change in the dual process at the core of our method, and therefore significant further difficulties in the current proof, it is thus left as an open problem at this point. Another natural question is that of fluctuations around equilibrium for non-reversible boundary dynamics. Within the current state of the art however, deriving the equilibrium fluctuations requires sharp estimates on the $2$-point space-time correlation field   at the boundaries of the system, and it is not clear whether such estimates actually hold. In any case, such an estimate is not achievable with the technique employed in the current work.

As does the one laid out in \cite{DMP12}, our approach works in the case $\theta\geq1$, although technical changes would be required to account for the modified boundary conditions. Let us comment briefly on the the boundary densities derived in \eqref{eq:alpha}. In the case of a single reservoir with $0<\theta<1$, the current due to reservoir interaction is on a faster time scale than the particle current from the boundary to the bulk, so that the boundary of the system thermalizes immediately at the density which makes the reservoir current vanish, i.e. that of the reservoir. In our case, site $2$ is affected by a non-reversible dynamics, but the overall boundary dynamics is still faster than the bulk dynamics, so that the system immediately enforces a boundary density at which the reservoir and non-reversible current cancel each other out. This condition is precisely given by \eqref{eq:alpha}. In the case $\theta=1$, the boundary currents operate on the same time scale as the bulk current, so that the boundary density takes the value as with the stirring current cancels out the two others. Finally, for $\theta>1$, the boundary dynamics is slow, so that in the hydrodynamic limit, the boundary stirring currents cancel out, thus yielding Neumann boundary conditions.

\section{Model and main result}

Let $N$ be a scaling parameter and  denote by  $\Lambda_N=\{1, \ldots, N-1\}$ the bulk of the system.   We consider a continuous time Markov process $\{\eta_t:\,t\geq{0}\}$, with state space $\Omega_N:=\{0,1\}^{\Lambda_N}$. We denote $\eta$ the configurations, i.e. elements of  $\Omega_N$ where $\eta(x)=0$ means that the site $x$ is vacant while $\eta(x)=1$ means that the site $x$ is occupied. The process $\{\eta_t:\,t\geq{0}\}$ is driven by the infinitesimal generator  
\begin{equation}\label{generator_ssep}
	\mc L_{N}=N^2\mc L_{N,0}+N^{2-\theta}\pa{\mc L_{N, l}+\mc L_{N, r}},
\end{equation}
where $\theta\in(0,1)$. The parts of this generator act on functions $f:\Omega_N\rightarrow \bb{R}$ as
\begin{equation}\label{eq:gen}
	(\mc L_{N,0}f)(\eta)=
	\sum_{x=1}^{N-2}\Big(f(\eta^{x,x+1})-f(\eta)\Big)
\end{equation}
and 
\begin{align*}
	(\mc L_{N,l}f)(\eta)&=
	c_{l,1}(\eta)\Big(f(\eta^{1})-f(\eta)\Big)+	c_{l,2}(\eta)\Big(f(\eta^{2})-f(\eta)\Big),\\
	(\mc L_{N,r}f)(\eta)&=
	c_{r,1}(\eta)\Big(f(\eta^{N-1})-f(\eta)\Big)+c_{r,2}(\eta)\Big(f(\eta^{N-2})-f(\eta)\Big)
\end{align*} 
with rates
\begin{align*}
c_{l,1}(\eta)&=r \cro{\bar \rho (1-\eta(1))+(1-\bar \rho)\eta(1)},\\
c_{r,1}(\eta)&=r' \cro{\bar \rho'(1-\eta(N-1))+(1-\bar \rho')\eta(N-1)},\\
c_{l,2}(\eta)&=c\cro{\eta(1)(1-\eta(2))+(1-\eta(1))\eta(2)}+b\eta(1)(1-\eta(2)),\\
c_{r,2}(\eta)&=c'\cro{\eta(N-1)(1-\eta(N-2))+(1-\eta(N-1))\eta(N-2)}+b'\eta(N-1)(1-\eta(N-2)).
\end{align*}
In the formulation above, we shortened

\begin{equation*}
\eta^{x,y}(z) = 
\begin{cases}
\eta(z), \; z \ne x,y\\
\eta(y), \; z=x\\
\eta(x), \; z=y
\end{cases}
, \quad \mbox{ and } \quad \eta^x(z)= 
\begin{cases}
\eta(z), \; z \ne x\\
1-\eta(x), \; z=x
\end{cases}.
\end{equation*}
For convenience, we reparametrized the boundary generators w.r.t. Part I of this article \cite{EGN}. To get back to the notations from \cite{EGN}, one can define 
\begin{align*}
	&\alpha_1=r\bar\rho,\quad\gamma_1=r(1-\bar\rho),\quad
	\alpha_2=b+c,\quad \gamma_2=c,\\ 
	&\beta_1=r'\bar\rho',\quad\delta_1=r'(1-\bar\rho'),\quad
	\beta_2=b'+c',\quad\delta_2=c'.
\end{align*}
We assumed (using Part 1's notations) that $\alpha_2\geq \gamma_2$ and $\beta_2>\delta_2$. This is purely for convenience:  if for example, $\alpha_2\leq \gamma_2$,  one would merely switch $\alpha_2=c$, $\gamma_2=b+c$,  for which our entire proof holds as well. Throughout this article, $r$, $b$, $c$, $r'$, $b'$, $c'$ and $\theta$ are constant, and even without specific mention, all quantities can depend on them. With these notations, we can reinterpret the boundary generator as follows: at rate $r$ (resp. $r'$), site $1$  (resp. $N-1$) is replaced by a Bernoulli with {parameter $\bar \rho$ (resp. $\quad  \bar \rho'$).} At rate $c$ (resp. $c'$), site $2$ (resp. $N-2$) becomes a copy of site $1$ (resp. $N-1$), and finally, at rate $b$ (resp. $b'$), site $2$ (resp. $N-2$) is filled by a particle if site $1$ (resp. $N-1$) is occupied, otherwise it is left unchanged.

\medskip

The case $\theta\geq 1$ was treated in \cite{EGN}. We consider here the case where $\theta<1$, i.e. the case where the reservoirs are strong. The boundary dynamics we chose does not admit product measures as stationary states, and because $ \theta<1$ they are not slowed down enough for the usual entropy estimates to hold. Instead, as in \cite{ELX18}, \cite{E18}, we exploit duality estimates on random walks, to write both the discrete profile and the two-point  correlation function as solutions of ``artificially'' closed equations, and then estimate the boundary terms needed to close them. Our main assumption is
\begin{equation}
\label{ass:thetal1}
\tag{H1}
b<r\quad  \mbox{ and  }\quad b'<r',
\end{equation}
which is analogous to assumption (2.13) in \cite{ELX18}. {As briefly described in the introduction, this assumption is made in order for the dual branching process defined in Section \ref{unknowns} to eventually die.}
Define $\alpha$ and $\alpha'$ as the unique solutions in $[0,1]$ of the equations
\begin{equation}
\label{eq:alpha}
r(\bar\rho-\alpha) +b\alpha(1-\alpha)=0\quad \mbox{and } \quad r'(\bar\rho'-\alpha') +b'\alpha'(1-\alpha')=0.
\end{equation}

Note that we have the explicit expressions
\[\alpha=f(r, b, \bar\rho), \quad \alpha'=f(r', b', \bar\rho'), \quad \mbox{ where }\quad f(r,b,\bar\rho)=\frac{\sqrt{(r-b)^2+4br\bar\rho}+b-r}{2b},\]
{and that the parameters $c$, $c'$ play no role in these definitions.}
\medskip

 Fix an initial profile  
 $f_0:[0,1]\rightarrow[0,1]$ in $C^1([0,1])$, and start the process  $\{\eta_t:\,t\geq{0}\}$ from the product measure 
 \[\mu _N(\eta)=\prod_{x\in \Lambda_N}\Big[f_0(x/N)\eta(x)+(1-f_0(x/N))(1-\eta(x))\Big]\] 
 fitting the initial profile $f_0$. We denote  by $\PP_{\mu _{N}}$ the distribution of the process   $\{\eta_t:\,t\geq{0}\}$ started from the distribution $\mu_N$ and with infinitesimal generator given by \eqref{generator_ssep}.
We are now ready to state our main result.

\begin{thm}[Hydrodynamic limit]
\label{th:hyd_thetaless1}
Under assumption \eqref{ass:thetal1}, for any $t\geq 0$, any Riemann-integrable function $G$, and  any $\delta>0$
\begin{equation*}\label{limHidreform2}
 \lim _{N\to\infty } \PP_{\mu _{N}}\pa{ \left\vert \dfrac{1}{N-1}\sum_{x \in \Lambda_{N} }G\left(\tfrac{x}{N} \right)\eta_{t}(x) - \int_{0}^1G(u)\rho_{t}(u)du \right\vert    > \delta }= 0,
\end{equation*}
where  $\rho_{t}(\cdot)$ is the unique classical solution of the heat equation with Dirichlet boundary conditions $\alpha$ and $ \alpha'$ solutions of \eqref{eq:alpha}:
\begin{equation}
\label{eq:PDEinf1}
 \begin{cases}
 \partial_{t}\rho_{t}(u)=  \Delta\, {\rho} _{t}(u)& \mbox{for }  t\geq 0,\; u\in(0,1)\\
 \rho _{t}(0)= \alpha  &\mbox{for }   t\geq 0\\
 \rho_t(1)= \alpha' &  \mbox{for }  t\geq 0\\
 \rho_0(\cdot)= f_0(\cdot)&
 \end{cases}.
\end{equation}
\end{thm}
\begin{rem}[Initial distribution]
We chose the initial distribution to be a product measure fitting a smooth initial profile. This is mainly not to burden the proof, which would hold as well assuming that the correlations of the initial distribution decay uniformly, 
\[\limsup_{\ell\to\infty}\limsup_{N\to\infty}
\sup_{\substack{ x,y\in \Lambda_N \\
|x-y|\geq \ell}}\mu_N\cro{(\eta(x)-f_0(x))( \eta(y)-f_0(y))}=0.\]
The regularity assumption on the initial profile $f_0$ could also be weakened, but that would entail significant extra technical difficulties to prove Lemma \ref{lem:grad72} below. 
\end{rem} 

\begin{rem}[On assumption \eqref{ass:thetal1}]
As was already mentioned, Assumption \eqref{ass:thetal1} is analogous to assumption (2.13) in \cite{ELX18}. The reason for this assumption is the following. In order to determine the value of $\eta_t$ at the boundaries, we are able to write $\eta_t(3)$ as a function of its past, by following the sites that had an influence over its value. This allows us to write $\eta_t(3)$ as a function of a branching process, which branches at the boundaries at rate $b$, $b'$ and dies at rate $r$, $r'$. Assumption \eqref{ass:thetal1} ensures that this branching process ultimately dies out completely, and that we are able to determine the value of $\eta_t(3)$.

\end{rem}

\begin{rem}[The case $\theta=0$]
Note that under an assumption analogous to (2.13) in \cite{ELX18}, the case $\theta=0$ is a consequence of  \cite{ELX18}. However, unlike in the present article, in the case $\theta=0$ correlations are introduced in the boundary dynamics, so that we no longer have an explicit expression for the boundary densities $\alpha$ and $\alpha'$. This is due to the fact that the distribution of the determination tree (cf. Section \ref{label:determination}) is no longer close to a Galton-Watson distribution, i.e. Lemma \ref{lem:tree} no longer holds.
\end{rem}

\begin{rem}[Heuristics on \eqref{eq:alpha}]
Roughly speaking, \eqref{eq:alpha} is an equilibrium formula for the boundary dynamics, in which the stirring jump rate is infinite (which is formally the case in the limit $N\to\infty$, since the boundary dynamics is slowed down). Unlike in the case $\theta=0$, the boundary sites decorrelate, and formally letting $\alpha=\bE(\eta_1)=\bE(\eta_2)=\bE(\eta_3)=\dots$ in the limit $N\to\infty$, and then using  $\bE(\mc L_{N}\eta_1)=\bE(\mc L_{N}\eta_2)=0$ yields \eqref{eq:alpha}. This decorrelation at the boundary, due to the slowed down dynamics, is the reason why an explicit formula for the boundary density can be achieved.
\end{rem}

\begin{rem}[General jump rates]
The proof we present for Theorem \ref{th:hyd_thetaless1} is not specific to these jump rates. It would actually hold for any perturbation of the flipping dynamics considered in \cite{ELX18}, (under an analogous assumption to Equation (2.13) in \cite{ELX18}). However, in order not to burden the proof and for consistency w.r.t. \cite{EGN}, we choose the boundary generator given in \eqref{eq:gen}. 
Note that unlike in \cite{ELX18}, because of the slowed down boundary dynamics, an explicit expression for the boundary densities can be derived, however due to the construction of our dual process and its associated tree (cf. Section \ref{label:determination}), obtaining an explicit formula in the general case would prove burdensome.
\end{rem}

The Markov chain induced by the generator $(\mc L)_N$  defined in \eqref{generator_ssep} is irreducible on $\Omega_N$. We will denote by $\mu_N^{ss}$ its unique stationary state. Given the duality techniques used to prove Theorem \ref{th:hyd_thetaless1}, the hydrostatic limit for this model is a straightforward adaption of the hydrodynamic limit, we state it but will not prove it. Instead, we refer the reader to \cite{ELX18} for more details. 
\begin{thm}[Hydrostatic limit]
\label{th:hydrostat}
Under assumption \ref{ass:thetal1}, for  any Riemann integrable function $G$ and  any $\delta>0$
\begin{equation*}\label{limHidreform2}
 \lim _{N\to\infty } \mu^{ss} _{N}\cro{ \left\vert \dfrac{1}{N-1}\sum_{x \in \Lambda_{N} }G\left(\tfrac{x}{N} \right)\eta(x) - \int_{0}^1G(u)\rho^*(u)du \right\vert    > \delta }= 0,
\end{equation*}
where  $\rho^*(\cdot)$ is harmonic with boundary conditions $\alpha$ and $ \alpha'$ solutions of \eqref{eq:alpha}, i.e.:
\begin{equation}
\label{eq:PDEinf2}
 \begin{cases}
 \Delta\, {\rho^*}=0, \\
 \rho^* (0)= \alpha  \quad \mbox{ and  } \quad  \rho^* (1)= \alpha',
 \end{cases}
\end{equation}
so that $\rho^*(u)$ is linear connecting $\alpha$ and $\alpha'$.
\end{thm}

 Throughout, we shorten 
\begin{equation}
\label{eq:Defetheta}
\widehat{\theta}:=(1-\theta)/2\quad>0. 
\end{equation}
For $x,y \in \Lambda_N$, $t\geq 0 $, we further introduce the discrete density profile
\begin{equation}
\label{eq:densityfield}
\rho_t^N(x)=\mathbb{E}_{\mu_N}\cro{\eta_{t}(x)}
\end{equation}
 and the two-point correlation function
\begin{equation} 
\label{eq:Defphi}
\varphi_t^N(x,y)=\mathbb{E}_{\mu_N}\cro{\pa{\eta_{t}(x)-\rho_t^N(x)}\pa{\eta_{t}(y)-\rho_t^N(y)}}.
\end{equation}

The article is organized as follows.
To derive the hydrodynamic equation, as in previous works \cite{ELX18,E18}, we estimate the discrete density profile (Section \ref{sec:densityfield}) and the two-point correlation (Section \ref{sec:corr}) function, namely $\rho_t^N$ and $\varphi_t^N$ by writing each one of them as a solution of a discrete difference equation with unknown boundary terms.  One of the main difficulties w.r.t. \cite{ELX18,E18} is to estimate those boundary terms, which require to refine the construction because of the slowed down boundary dynamics. Once this is done, proving the hydrodynamic limit is straightforward (cf.  Section \ref{sec:hydrolimit}). The estimation of the boundary term is the purpose of Section \ref{sec:branching}.  We will refer to  \cite{ELX18,E18} when the results are analogous, and detail the new contributions of this article.

\section{Density field}
\label{sec:densityfield}

In this section, we write the discrete profile $\rho_t^N$ defined in \eqref{eq:densityfield} as an approximation of the solution of \eqref{eq:PDEinf1}. More precisely, we have the following:
\begin{prop}
\label{prop:discapprox}
For any $t\geq 0$, and any Riemann integrable function $G\in C^2([0,1])$, 
\[
\limsup_{N\to\infty}\left(\frac{1}{N}\sum_{x=3}^{N-3}G(x/N)\rho^N_t(x)-\int_{[0, 1]}G(u)\rho_t(u)du\right)^2=0.
\]
\end{prop}
The main ingredient to prove Proposition \ref{prop:discapprox} is the following Lemma, whose proof is postponed to Section \ref{sec_detfinal} because it requires significant work. 
\begin{lem}
\label{lem:res}
Recall from Equation \ref{eq:alpha} the definition of $\alpha$ and $\alpha'$, there exists $\varepsilon>0$ such that 
\[\sup_{t\geq N^{-\widehat{\theta}}}|\rho_t^N(3)-\alpha|=O(N^{-\varepsilon})\quad  \mbox{ and  }\quad \sup_{t\geq N^{-\widehat{\theta}}}|\rho_t^N(N-3)-\alpha'|=O(N^{-\varepsilon}). \]
\end{lem}
The first identity is a consequence of Corollary \ref{prop:valL} below. The second identity being strictly analogous, we will not prove it. Note that this result fixes the Dirichlet boundary condition of the discretized density profile $\rho_t^N$. We assume for now that Lemma \ref{lem:res} holds, and with it we will prove Proposition \ref{prop:discapprox}
\begin{proof}[Proof of Proposition \ref{prop:discapprox}]
We claim that  for any $t>0$, there exists  $\varepsilon>0$ such that 
\begin{equation}
\label{eq:gamma}
\sup_{x\in\{3,\dots,N-3\}}|\rho^N_t(x)-\rho_t(x/N)|=O(N^{-\varepsilon}).
\end{equation}
Applying Dynkin's identity  $\partial_t \bE_{\mu_N}\cro{f(\eta_t)}=\bE_{\mu_N}\cro{\mathcal{L}_Nf(\eta_t)}$, to $f(\eta)=\eta(x)$ for any $x\in \{4, \cdots,N-4\}$, elementary computations yield that $\rho_t^N(\cdot)$ is solution of the discrete difference diffusion equation
\begin{equation}
\label{eq:discrho}
\begin{cases}
\partial_t \rho_t^N=\Delta_N\rho_t^N, & \mbox{ for }4\leq x \leq N-4\\
\rho_t^N(3)=\alpha_{t,N},\\
\rho_t^N(N-3)=\alpha_{t,N}',\\
\rho_0^N=\rho_0(\cdot/N) 
\end{cases}
,\end{equation}
where we defined $\alpha_{t,N}=\rho^N_t(3)$ and $\alpha_{t,N}'=\rho_t^N(N-3)$, and 
\[\Delta_N\rho^N(x)=N^2\cro{\rho^N(x+1)+\rho^N(x-1)-2\rho^N(x)}\]
is the discrete Laplacian.  Note that the two identities at the space boundaries are trivial, and ``artificially'' close the equation satisfied by $\rho^N_t$.

Denote by $X_s$ (resp. $B_s$) a continuous time random walk on $\Z$ jumping at rate $N^2$ to each of its neighbors (resp. standard Brownian Motion on $\bR$). We denote by $H^X$ (resp. by $H^B$) the hitting time of the boundary $\{3, N-3\}$ (resp. $\{0,1\}$), and shorten $H^X_t=t\wedge H^X$ (resp. $H^B_t=t\wedge H^B$).  We denote by $\bE^X_x$, and $\bE^B_u$ the expectations w.r.t. the distributions of both $X$ and $B$ starting from $x$ and $u$, respectively.  By Feynman-Kac's formula, we can write for $x\in \{3,\dots,N-3\}$, $u\in [0,1]$  
\[\rho^N_t(x)=\bE^X_x\cro{\rho^N_{t-H^X_t}(X_{H^X_t})} \quad \mbox{ and } \quad \rho_t(u)=\bE^B_{u}\cro{\rho_{t-H^B_t}(B_{H^B_t})} .\]
We can therefore write $|\rho^N_t(x)-\rho_t(x/N)|$ as the sum of three contributions, the first one corresponding to the case where $X_{H_t^X}=3$ (resp. $B_{H_t^B}=0$), the second $X_{H_t^X}=N-3$ (resp. $B_{H_t^B}=1$), and the last one to the case $H_t^X=t$ (resp. $H_t^B=t$). We write these three contributions as 
\[C_l(t,x):=\left|\bE^X_{x}\cro{\rho^N_{t-H^X_t}(3){\bf 1}_{\{X_{H_t^X}=3\}}}-\alpha\bP^B_{x/N}\left(B_{H_t^B}=0\right)\right|,\]
\[C_r(t,x):=\left|\bE^X_{x}\cro{\rho^N_{t-H^X_t}(N-3){\bf 1}_{\{X_{H_t^X}=N-3\}}}-\alpha'\bP^B_{x/N}\left(B_{H_t^B}=1\right)\right|,\]
\[C^*(t,x):=\left|\bE^X_{x}\cro{\rho^N_0(X_t){\bf 1}_{\{H_t^X=t\}}}-\bE^B_{x/N}\left(\rho^N_0(B_t){\bf 1}_{\{H_t^B=t\}}\right)\right|.\]
Given theses notations, we have
\[|\rho^N_t(x)-\rho_t(x/N)|\leq C_l+C_r+C^*.\]
Recall \eqref{eq:Defetheta}. To estimate $C_l$ and $C_r$, we consider two cases, depending on whether $H_t^X$ is in $ [0,t-N^{-\widehat{\theta}}]$ or in $(t-N^{-\widehat{\theta}},t)$. In the first case, thanks to Lemma \ref{lem:res}, $\rho^N_{t-H^X_t}(3)=\alpha+O(N^{-\varepsilon})$.  More precisely, since $\rho^N$ is less than $1$ we  can write 
\[C_l\leq \left|\bP^X_{x}\left(X_{H_t^X}=3 \right)-\bP^B_{x/N}\left(B_{H_t^B}=0\right)\right|+2\bP^X_{x}\left( H_t^X\in (t-N^{-\widehat{\theta}},t) \right)+ O(N^{-\varepsilon}).\]
The second term on the right-hand side above can be estimated according 
to Lemma 3.3, p.12 of \cite{E18}, which yields 
\[\bP^X_{x}\left(H_t^X\in (t-N^{-\widehat{\theta}},t) \right)=O\left(N^{-\widehat{\theta}}t^{-3/2}\right) .\]
Finally, the first contribution , as well as $C^*$, can be estimated by using the Berry Essen inequality, i.e. approximating the random walk by a Brownian motion as in the proof of Lemma  3.3 of \cite{E18}, to obtain for a positive constant $\varepsilon'>0$
\[\left|\bP^X_{x}\left(X_{H_t^X}=3 \right)-\bP^B_{x/N}\left(B_{H_t^B}=0\right)\right|=O(N^{-\varepsilon'})\]
and
\[C^*(t,x)=O(N^{-\varepsilon'}).\] 
The bounds above can be taken independent of $x$. This proves \eqref{eq:gamma} and the proposition.
\end{proof}

Before turning to the two-point correlation function, we state the following technical Lemma, which will be needed in what follows.
\begin{lem}
\label{lem:grad72}
For any $\delta>0$,
\[\sup_{\substack{ \delta N\leq x \leq (1-\delta)N\\t\geq 0}}|\rho_t^N(x+1)-\rho_t^N(x)|\leq O\pa{ (\delta N)^{-1}\log N}.\]
\end{lem}
\begin{proof}[Proof of Lemma \ref{lem:grad72}]
Fix $\delta>0$, and consider once again a random walk $X_s$ on $ \Z$, for which we denote $H_{\{3, N-3\}}$ the hitting time of the boundaries x=3, N-3. Recall that $\rho_t^N(x)$ is solution of  \eqref{eq:discrho}, choose a sequence $a_N$ that vanishes as $N\to\infty$, and define the stopping time 
\[\tau=\tau_{\delta, N}:=\delta^2N^2a_N\wedge H_{\{3,N-3\}}.\] 
By the Feynman-Kac formula, we can write for any $3\leq x \leq N-3$
\[\rho_t^N(x)=\bE_x\left[\rho^N_{t-t\wedge \tau}(X_{t\wedge \tau})\right].\]
A symmetric random walk would typically require a time of order $\delta^2N^2$ to move on a distance $\delta N$, therefore there exists a universal constant $K$ such that  for any $\delta N\leq x \leq (1-\delta)N$, 
\[\bP(\tau=H_{\{3,N-3\}})\leq e^{-K /a_N}.\]
In particular, one can write  from the identity above
\[\rho_t^N(x)=\sum_{y=3}^{N-3}\bE_x\left[\rho^N_{t-t\wedge \tau}(y){\bf 1}_{\{X_{t\wedge \tau}=y, \tau= \delta^2N^2a_N\}}\right]+O(e^{-K /a_N}).\]
In particular, one obtains for $t\leq \delta^2N^2 a_N$
\[|\rho_t^N(x+1)-\rho_t^N(x)|=\frac{1}{N}||\partial_uf_0||_{\infty}+O(e^{-K /a_N}).\]
and for $t\geq \delta^2N^2a_N$
\begin{align*}
|\rho_t^N(x+1)-\rho_t^N(x)|\leq &\sum_{y=3}^{N-3}|\bP_{x+1}\left[X_{\delta^2N^2a_N}=y\right]-\bP_{x}\left[X_{\delta^2N^2a_N}=y\right]|+O(e^{-K  /a_N})\\
=&\sum_{y=3}^{N-3}|\bP_{y}\left[X_{\delta^2N^2a_N}=x+1\right]-\bP_{y}\left[X_{\delta^2N^2a_N}=x\right]|+O(e^{-K /a_N})\\
=&\sum_{y=x+3-N}^{x-3}|\bP_{0}\left[X_{\delta^2N^2a_N}=y+1\right]-\bP_{0}\left[X_{\delta^2N^2a_N}=y\right]|+O(e^{-K /a_N}).\end{align*}
Since $\bP_{0}\left[X_{\delta^2N^2a_N}=y\right]$ is maximal in $0$ and decreasing in $|y|$, most of the terms in the estimate above cancel out, and we obtain 
\[|\rho_t^N(x+1)-\rho_t^N(x)|\leq 2\bP_{0}\left[X_{\delta^2N^2a_N}=0\right]+O(e^{-K /a_N}).\]
By approximating the random walk by a Brownian motion, one shows straightforwardly (cf. Equations (3.2) and (3.5) in \cite{DMP12})  $\bP_{0}\left[X_{\delta^2N^2a_N}=0\right]\leq\pa{\sqrt{2\pi a_N}\delta N}^{-1}$, which, letting $a_N=K/\log N$, proves the result.
\end{proof}

\section{Estimation on the bulk two-point correlation function}
\label{sec:corr}
We now estimate the two-point correlation function defined in \eqref{eq:Defphi}.
\begin{prop}
\label{prop:correlations}
For any positive $t$, and any fixed $\delta>0$
\[
\limsup_{N\to\infty}\sup_{\substack{ 3\leq x\leq (1-\delta)N-1\\ \delta N+1\leq y\leq  N-3}}|\varphi_t(x,y)|=0.
\]
\end{prop}
Once again, the main ingredient to prove this result is an estimate of the correlation function at the boundary of the domain above, given by the following Lemma.

\begin{lem}
\label{lem:bcor}
There exists $\varepsilon$ such that for any $\delta>0$,
\[\sup_{\substack{\delta N \leq y\leq N-1\\
t\geq 0}}|\varphi_t^N(3,y)|=O(N^{-\varepsilon}) \quad  \mbox{ and  }\quad  \sup_{\substack{1 \leq x\leq N(1-\delta)\\
t\geq 0}}|\varphi_t^N(x,N-3)|=O(N^{-\varepsilon}) ,\]
where the constant in $O(N^{-\varepsilon})$ can depend on $\delta$.
\end{lem}
The proof of this Lemma is also postponed  Section \ref{sec_detfinal}. For now, we assume it holds and prove Proposition \ref{prop:correlations}.
\begin{proof}[Proof of Proposition \ref{prop:correlations}]
Define the set, represented in Figure \ref{fig:zozo}
\begin{figure}
\includegraphics[width=12cm]{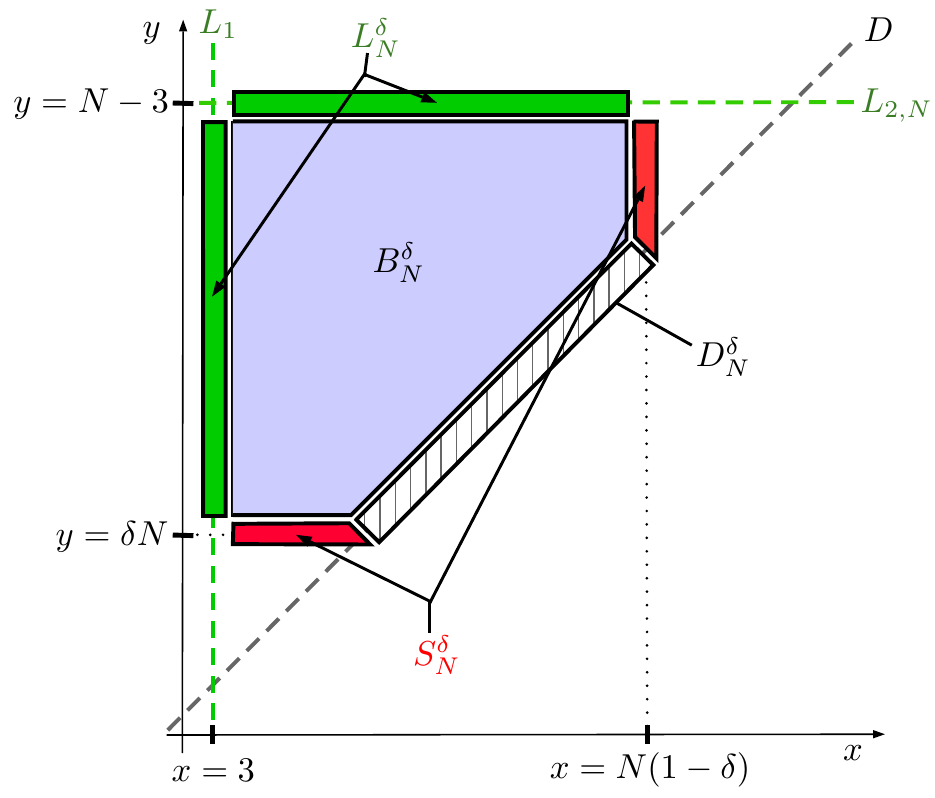}
\caption{}
\label{fig:zozo}
\end{figure}

\[B_N^{\delta}=\Big\{(x,y)\in \llbracket 3,(1-\delta)N-1\}\times \llbracket\delta N+1, N-3\rrbracket, \mbox{ such that } x<y-1 \Big\}.\]
We also define the boundary sets 
\[L_1^{\delta}=\Big\{(3,y) \mbox{ for } y\in \llbracket\delta N, N-3\rrbracket \Big\},\]
\[L_{2,N}^{\delta}=\Big\{(x,N-3)\mbox{ for } x\in \llbracket 3,(1-\delta) N\rrbracket \Big\},\]
\[S_{1,N}^{\delta}=\Big\{((1-\delta)N, y ) \mbox{ for }y \in \llbracket(1-\delta)N+1, N-4\rrbracket \Big\},\]
\[S_{2,N}^{\delta}=\Big\{(x,\delta N) \mbox{ for } x\in \llbracket 4,\delta N -1\rrbracket \Big\},\]
\[D_N^{\delta}=\Big\{(x,x+1) \mbox{ for } x \in \llbracket \delta N, (1-\delta) N\rrbracket \Big\}.\]
Finally, shorten $L_N^{\delta}=L_1^{\delta}\cup L_{2,N}^{\delta}$, $S_N^{\delta}=S_{1,N}^{\delta}\cup S_{2,N}^{\delta}$, and $\partial B_N^{\delta}=L_N^{\delta}\cup S_N^{\delta}$.
Using once again Dynkin's identity $\partial_t \bE_{\mu_N}[f(\eta_t)]=\bE_{\mu_N}[\mathcal{L}_Nf(\eta_t)]$ to $f(\eta)=\eta(x)\eta(y)$ for any $(x,y)\in B_N^\delta$, and since the initial measure for the process is a product measure, elementary computations yield that $\varphi_t^N$ is solution of the discrete difference diffusion equation
\begin{equation}
\label{eq:discphi}
\begin{cases}
\partial_t \varphi_t^N=\Delta_{2,N}\varphi_t^N& \mbox{ for }(x,y)\in  B_N^{\delta}\\
\partial_t \varphi_t^N=\nabla_{2,N}\varphi_t^N +g_{t,N} & \mbox{ for }(x,y)\in  D_N^{\delta}\\
\varphi_t^N=\phi_{t,N}&\mbox{ for }(x,y)\in \partial B_N^{\delta}\\
\varphi_0^N\equiv 0
\end{cases}
.\end{equation}
The operator $\Delta_{2,N}$ stands for the two-dimensional discrete Laplacian 
\[\Delta_{2,N}\varphi_t^N(x,y)=N^2[\varphi_t^N(x+1, y)+\varphi_t^N(x-1, y)+\varphi_t^N(x, y+1)+\varphi_t^N(x, y-1)-4\varphi_t^N(x, y)],\]
the operator $\nabla_{2,N}$ represents the reflection at the diagonal  $D_N^{\delta}$
\[\nabla_{2,N}\varphi_t^N(x,x+1)=N^2[\varphi_t^N(x-1, x+1)+\varphi_t^N(x, x+2)-2\varphi_t^N(x, x+1)],\]
and the function $g_{t,N}$ is defined on $ D_N^{\delta}$ as 
\[g_{t,N}(x,x+1)=- N^2(\rho_t^N(x+1)-\rho_t^N(x+1))^2.\]
We also defined $\phi_{t,N}(x,y)=\varphi^N_t(x,y)$ on $\partial B_N^{\delta}$. Note that once again, the boundary identity is trivial, and ``artificially'' closes the equation satisfied by $\varphi_t^N$. 

\medskip

Consider now  a random walk $X_t$ on $\Z^2$ with generator  $ \nabla_{2,N} $ on 
\[D:=\{(x,x+1), x\in \Z\}\]
and $\Delta_{2,N}$ on $\Z^2\setminus D$. Defined in this way, $X_t$ is a simple symmetric random walk on $\Z^2$ reflected at the diagonal $D$. For any set $E\subset \Z^2$, we denote by $H(E)$ the hitting time of the set $E$, $H(E)=\inf\{t\geq 0, \;X_t\in E\}$, and shorten $H_t(E)=t\wedge H(E)$. Finally, we denote by $\bP_{x,y}$ the distribution of the random walk started from $(x,y)$ and $\bE_{x,y}$ the corresponding expectation. Then, as a consequence of \eqref{eq:discphi} and Feynman-Kac's formula, one obtains for any $t\geq 0$ and $(x,y)\in B_N^{\delta}\cup D_N^\delta $
\begin{equation}
\label{eq:phi}
\varphi_t(x,y)=\bE_{x,y}\cro{\varphi_{t-H_t(\partial B_N^{\delta})}(X_{H_t(\partial B_N^{\delta})})+\int_0^{H_t(\partial B_N^{\delta})}g_{t-s,N}(X_s){\bf 1}_{\{X_s\in D_N^{\delta}\}}ds}.
\end{equation}
Before $H_t(\partial B_N^{\delta})$, the random walk $X_t$ cannot reach $D\setminus D_N^{\delta}$, and $H_t(\partial B_N^{\delta})\leq t$, so that according to Lemma \ref{lem:grad72} the absolute value of the second term above is less or equal than 
\[\bE_{x,y}\cro{\int_0^{H_t(\partial B_N^{\delta})}g^N_{t-s}(X_s){\bf 1}_{\{X_s\in D_N^{\delta}\}}ds}\leq \pa{\frac{C\log N}{\delta}}^2\bE_{x,y}\cro{\int_0^{H_t(\partial B_N^{\delta})}{\bf 1}_{\{X_s\in D_N^{\delta}\}}ds}\]
for some constant $C$. The right hand side can be estimated using the same steps in Section 4.4 of \cite{E18}, after which one obtains straightforwardly that 
\[\bE_{x,y}\cro{\int_0^{H_t(\partial B_N^{\delta})}{\bf 1}_{\{X_s\in D_N^{\delta}\}}ds}=O( N^{-1}\log N).\]
We can now write, using \eqref{eq:phi}, that 
\begin{equation}
\label{eq:phi2}
\varphi_t(x,y)=\bE_{x,y}\cro{\varphi_{t-H_t(\partial B_N^{\delta})}(X_{H_t(\partial B_N^{\delta})})}+O(N^{-1}(\log N)^3\delta^{-2}).
\end{equation}
The function $\varphi$ is uniformly bounded by $1$, vanishes at time $0$, and according to Lemma \ref{lem:bcor} 
\[\sup_{\substack{(x,y)\in L_N^\delta\\ t\geq0}}\varphi_t(x,y)=O(N^{-\varepsilon}),\] 
therefore \eqref{eq:phi2} yields for any $t\geq 0$
\begin{align}
\label{eqphihit}
|\varphi_t(x,y)|&\leq\bP_{x,y}\pa{H_t(\partial B_N^{\delta})=H(S^{\delta}_N)}+ O(N^{-\varepsilon})+O(N^{-1}(\log N)^3\delta^{-2})\nonumber\\
&\leq\bP_{x,y}\pa{H(L^\delta_N)>H(S^\delta_N)}+ O(\delta^{-2}N^{-\varepsilon}).
\end{align}

Define $L_N=L_{1}\cup L_{2,N}$, where $L_{1}=\{(3,y), \mbox{ for } y\in \Z\}$ and $L_{2,N}=\{(x,N-3) \mbox{ for } x\in \Z\}$  are represented in Figure \ref{fig:zozo}. Then one easily checks that 
\begin{figure}
\includegraphics[width=9cm]{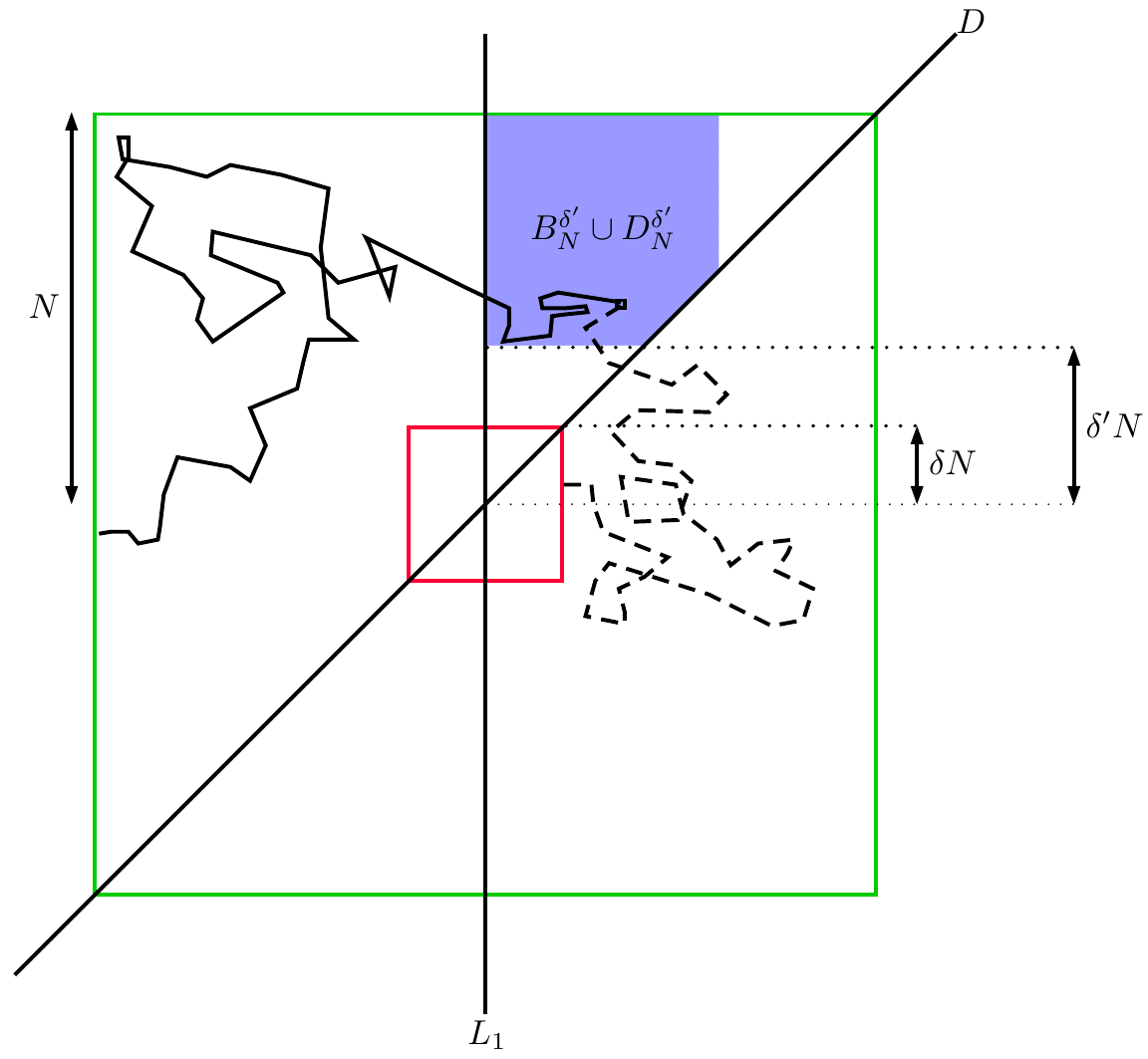}

\caption{A random walk started in $B_N^{\delta'}$ (blue area) has a probability of order $\log(\delta')/\log(\delta)$ of hitting $\Lambda_{\delta N}$ (red boundary) before $\Lambda_N^c$ (green boundary) (dashed trajectory). In particular, as $\delta\to 0$, it will most likely hit the green boundary before the red one (solid trajectory).}
\label{fig:trajectories}
\end{figure}
\begin{equation}
\label{eq:reflections}
\bP_{x,y}\pa{H(L^\delta_N)>H(S^\delta_N)}\leq \sum_{i=1}^2\bP_{x,y}\pa{H(L_N)>H(S^\delta_{i,N})}.
\end{equation}
Both terms in the right hand side above are treated in the same way. Assume, for example, that $i=1$. Then, 
\[\bP_{x,y}\pa{H(L_N)>H(S^\delta_{1,N})}\leq \widetilde{\bP}_{x,y}\pa{H(L_N)>H(S^\delta_{1,N})},\]
where $\widetilde{\bP}_{x,y}$ is the distribution of a random walk $ \widetilde{X}$  is \emph{reflected} both at $D$ and at $L_{2,N}$. Note that a reflected random walk can be mapped into a non-reflecting one by, each time the random walks hits the reflecting line, choosing with probability $1/2 $ a side of the line to carry on with the random walk. Consider $\delta'>\delta$ and assume now that the starting point is in $(x,y)\in B_N^{\delta'}\cup D_N^{\delta'}$. Since the random walk is reflected both at the horizontal boundary $L_{2,N}$, as illustrated in Figure \ref{fig:trajectories}, thanks to the previous observation, the probability $\widetilde{\bP}_{x,y}\pa{H(L_1)>H(S^\delta_{1,N})}$ is less or equal than the probability that a random walk starting at a distance $\delta'N$ of the origin hits $\Lambda_{\delta N}=\{(x,y), x, y\leq \delta_N\}$ before hitting  $\Lambda_{N}^c$, so that for any $(x,y)\in  B_N^{\delta'}\cup D_N^{\delta'}$
\[\widetilde{\bP}_{x,y}\pa{H(L_1)>H(S^\delta_{1,N})}\leq C \frac{\log (\delta' N)-\log N }{\log (\delta N)-\log N}=C\frac{\log \delta' }{\log \delta }.\]
An analogous bound holds for the second term in \eqref{eq:reflections}, i.e. for any $(x,y)\in  B_N^{\delta'}\cup D_N^{\delta'}$,
\[\bP_{x,y}\pa{H(L^\delta_N)>H(S^\delta_{2,N})}\leq C\frac{\log \delta' }{\log \delta }.\]
Plugging this bound into \eqref{eqphihit} yields that for any $(x,y)\in B_N^{\delta'}$, $t\geq 0$, 
\begin{equation*}
|\varphi_t(x,y)|\leq C(\delta)O(N^{-\varepsilon})-\frac{C(\delta')}{\log(\delta)}.
\end{equation*}
so that letting $N\to \infty$ and then $\delta \to 0^+$ proves the result.
\end{proof}

\subsection{Proof of Theorem \ref{th:hyd_thetaless1}}
\label{sec:hydrolimit}
Propositions \ref{prop:discapprox} and \ref{prop:correlations} are enough to prove the hydrodynamic limit. Indeed, fix a smooth test function $G$ and a time $t>0$. We write 
\begin{multline}\bE_{\mu^N}\pa{\cro{\frac{1}{N}\sum_{x=1}^{N-1}G(x/N)\eta_t(x)-\int_{[0,1]}G(u)\rho_t(u)du}^2}\\
\leq 
2\bE_{\mu^N}\pa{\cro{\frac{1}{N}\sum_{x=1}^{N-1}G(x/N)(\eta_t(x)-\rho^N_t(x)}^2}\\
+2\cro{\frac{1}{N}\sum_{x=1}^{N-1}G(x/N)\rho^N_t(x)-\int_{[0,1]}G(u)\rho_t(u)du}^2.\end{multline}
Since $G$ is bounded, the first term on the right-hand side is, for any fixed $\delta'>0$, according to Proposition \ref{prop:correlations} bounded by $2\delta'+O(1/N)+o_N(1)$, so that letting $N\to\infty$  and then $\delta' \to 0$, it vanishes. The second term vanishes as well according to Proposition \ref{prop:discapprox}. Using Markov's inequality concludes the proof.

\medskip

\section{Branching process and estimation of the density at the boundary}
\label{sec:branching}

\subsection{Graphical construction and dual branching process}We now turn to Lemmas \ref{lem:res} and \ref{lem:bcor}, whose proofs were postponed and are presented here.
We follow analogous arguments and constructions to Sections 5.1 and 5.2 in \cite{ELX18}, although significant adaptations need to be made to take into  account the time dependency and the slowed down boundary dynamics. 

We start with a graphical construction of the dynamics. We consider $N+6$ independent Poisson point processes on $\R$ defined as follows :
\begin{itemize}
\item $N-2$ processes $\mathscr{N}_{x,x+1}$, $x\in \llbracket 1, N-2 \rrbracket$ with intensity $N^2$, representing the  exclusion  dynamics.
\item Two processes $\mathscr{N}^+_1$, $\mathscr{N}^+_{N-1}$ with respective intensities $N^{2-\theta}\bar \rho r$ and $N^{2-\theta}\bar \rho'r'$, representing putting a particle at a boundary site $1$, $N-1$ regardless of its previous value.
\item Two processes $\mathscr{N}^-_1$, $\mathscr{N}^-_{N-1}$ with respective intensities $N^{2-\theta}(1-\bar \rho)r$ and $N^{2-\theta}(1-\bar \rho')r'$ representing emptying a boundary site $1$, $N-1$ regardless of its previous value.
\item Two processes $\mathscr{N}^c$, $\mathscr{N}^{c'}$ with respective intensities $N^{2-\theta}c$ and $N^{2-\theta}c'$ representing the copy mechanism.
\item Two processes $\mathscr{N}^b$, $\mathscr{N}^{b'}$ with respective intensities $N^{2-\theta}b$ and $N^{2-\theta}b'$ representing the boundary sites $2$,$N-2$ being filled iff the other boundary site $1$, $N-1$ is occupied.
\end{itemize}
In order to build the process starting from a given configuration $\eta_0$, we order $0<t_1<t_2, \gab{\cdots}$ the points 
\[\{t_i, i\in \bN\setminus\{0\} \}=\pa{\mathscr{N}^\pm_1\cup \mathscr{N}^c\cup \mathscr{N}^b\cup \mathscr{N}^\pm_{N-1}\cup \mathscr{N}^{c'}\cup \mathscr{N}^{b'}\cup\pa{\cup_{x=1}^{N-2}\mathscr{N}_{x,x+1} }}\cap \R_+\]
of these processes, taken in increasing order starting from $0$. Note that since it has probability $0$, we discounted the possibility that two of these processes   contained the same point. We refer to  the points of the processes together with their type as \emph{marks}. For example, if $t\in \mathscr{N}^\pm_{N-1} $, (resp. $t\in \mathscr{N}^b\cup \mathscr{N}^{b'}$) we say that the  \emph{mark} at $t$ is of type $(\pm, N-1)$ (resp. of \emph{branching} type).  If $t\in \mathscr{N}_{x,x+1} $, we say that the  \emph{mark} at $t$ is of type $(x,x+1)$. 
For convenience, we store in a variable $ \boldsymbol{\mathcal{M}} $ the collection of all the marks up to some arbitrary time horizon $T$.

\medskip

With these Poisson point processes, given an initial configuration $\eta_0$, we are able to give a graphical construction of our Markov process $(\eta_t)_{t\geq 0}$. Define $t_0=0$ and $\eta_{t_0}=\eta_0$, and  ``solve'' the marks $t_1, t_2,\cdots$ one by one in the following way. For any $i>0$, the configuration remains constant on $[t_{i-1}, t_i)$. At time $t_i$, the configuration is changed depending on the type of the mark $t_i$. 
\begin{itemize}
\item If  $t_i\in \mathscr{N}_{x,x+1}$, we let $\eta_{t_i}=\eta^{x,x+1}_{t_{i-1}}.$
\item If $t_i\in \mathscr{N}^+_{1}$ (resp. $t_i\in  \mathscr{N}^+_{N-1}$), at time $t_i$, a particle is put at site $1$ (resp. $N-1$). We let $\eta_{t_i}(x)=\eta_{t_{i-1}}(x)$ for any $x\neq 1$ (resp. $x\neq N-1$), and  $\eta_{t_i}(1)=1$ (resp. $\eta_{t_i}(N-1)=1$).
\item If $t_i\in \mathscr{N}^-_{1}$ (resp. $t_i\in  \mathscr{N}^-_{N-1}$), at time $t_i$,  site $1$ (resp. $N-1$) is emptied. We let $\eta_{t_i}(x)=\eta_{t_{i-1}}(x)$ for any $x\neq 1$ (resp. $x\neq N-1$), and  $\eta_{t_i}(1)=0$ (resp. $\eta_{t_i}(N-1)=0$).
\item If $t_i\in \mathscr{N}^c$ (resp. $t_i\in  \mathscr{N}^{c'}$), at time $t_i$, site $2$ (resp. $N-2$) becomes a copy of site $1$ (resp. $N-1$), so that we let $\eta_{t_i}(x)=\eta_{t_{i-1}}(x)$ for any $x\neq 2$ (resp. $x\neq N-2$), and  $\eta_{t_i}(2)=\eta_{t_{i-1}}(1)$ (resp. $\eta_{t_i}(N-2)=\eta_{t_{i-1}}(N-1)$).
\item Finally, if $t_i\in \mathscr{N}^b_{2}$ (resp. $t_i\in  \mathscr{N}^{b'}$), at time $t_i$, we let $\eta_{t_i}(x)=\eta_{t_{i-1}}(x)$ for any $x\neq 2$ (resp. $x\neq N-2$), and  if $\eta_{t_{i-1}}(1)=1$ (resp. $\eta_{t_{i-1}}(N-1)=1$), we let $\eta_{t_i}(2)=1$ (resp. $\eta_{t_i}(N-2)=1$). Otherwise, we let $\eta_{t_i}(2)=\eta_{t_{i-1}}(2)$ (resp. $\eta_{t_i}(N-2)=\eta_{t_{i-1}}(N-2)$).
\end{itemize}
We leave it to the reader to check that this construction yields a Markov process $\eta_t$, with initial state $\eta_0$, and infinitesimal generator $N^2\mathcal{L}_N.$

\subsection{Set of unknown sites}
\label{unknowns}
Fix $x\in \Lambda_N$, and suppose that we want to determine the value of $\eta_t(x) $. To do so, we are going to explore the past of the process, and define a set $\hat{\mathcal{A}}(s)\subset\Lambda_N$, $s<t$  of unknown sites at time $s$ on which depends the value of $\eta_t(x) $. To do so, re-index the marks $t>t_{-1}>t_{-2}>\dots>t_{-\kappa}>0$ in the time interval $[0,t]$, we leave the set $\hat{\mathcal A}(s)$ constant on each of the intervals $[t_{-1},t]$, $[t_{-2}, t_{-1}), \dots$ until time $0$ is reached. For $k<0$, assume that $\hat{\mathcal{A}}(s)$ has been defined down to time $t_k$, i.e. on the whole segment $[t_k, t]$. To define $\hat{\mathcal{A}}(s)$ down to time $t_{k-1}$, we  define it at time $t_k^-$ and leave it constant down to time $t_{k-1}$. Shorten $A:=\hat{\mathcal{A}}(t_k)$ and $A':=\hat{\mathcal{A}}(t_k^-)$. We consider all the possible cases for the type of the mark $t_k$.
\begin{itemize}
\item If  $t_k\in \mathscr{N}_{x,x+1}$, $x\in A$ and $x+1\notin A$  (resp. $x+1\in A$ and $x\notin A$),  before the mark, the content of the unknown site $x$ (resp. $x+1$), particle or hole,  was in $x+1$ (resp. $x$), and we let $A'=A^{x,x+1}:=A\cup\{x+1\}\setminus\{x\}$ (resp. $:=A\cup\{x\}\setminus\{x+1\}$). Otherwise, if both or neither of the two sites are in $A$, we just let $A'=A$.
\item If $t_k\in \mathscr{N}^+_{1}$ (resp. $t_k \in  \mathscr{N}^+_{N-1}$), then at the mark  a particle was placed at site $1$  (resp. $N-1$). Therefore, if site $1$  (resp. $N-1$) was in the set $A$ of unknowns, it can be removed, and we let $A'=A\setminus\{1\}$ (resp.  $A'=A\setminus\{N-1\}$). Else, we let $A'=A$.
\item If $t_k\in \mathscr{N}^-_{1}$ (resp. $t _k\in  \mathscr{N}^-_{N-1}$), then at the mark,  site $1$  (resp. $N-1$) was emptied. Therefore, if site $1$  (resp. $N-1$) was in the set $A$ of unknowns, it is removed, and we let $A'=A\setminus\{1\}$ (resp.  $A'=A\setminus\{N-1\}$). Else, we let $A'=A$.
\item If $t_k \in \mathscr{N}^{c}$ (resp. $t_k\in  \mathscr{N}^{c'}$), then at the mark,  site $2$  (resp. $N-2$) became a copy of site $1$ (resp. $N-1$) so that if $2\in A$ (resp. $N-2\in A$), we let $A'=A\cup\{1\}\setminus\{2\}$ (resp. $A'=A\cup\{N-1\}\setminus\{N-2\}$). Else, we let $A'=A$.
\item Finally, if $t_k\in \mathscr{N}^b$ (resp. $t_k\in  \mathscr{N}^{b'}$), then at the mark $t_k$, something might have happened at site  $2$  (resp. $N-2$) if and only if site  $1$  (resp. $N-1$) was occupied. Else, nothing happened, so that we need to keep track of both of the sites $1$ and $2$ to know for sure the value at time $t$.  More precisely, if $2\in A$ (resp. $N-2\in A$), we then let $A'=A\cup \{1\}$ (resp.  $A'=A\cup \{N-1\}$). Else, we let $A'=A$.
\end{itemize}
We then  iterate the process to build $\hat{\mathcal{A}}(s)$ backwards in time. Recall that elements of $\hat{\mathcal{A}}$ are \emph{sites} whose value is unknown. To avoid ambiguity, we will refer to the elements of $\hat{\mathcal{A}}$ as \emph{flags}. Once time $0$ is reached, all the remaining flags in $\hat{\mathcal{A}}$ are determined by the initial configuration $\eta_0$, so that we let $\hat{\mathcal{A}}(0)=\emptyset$.
We say that a mark $t_k$ \emph{affected} $\hat{\mathcal{A}}$ if at least one of the sites concerned by the mark was in $ \hat{\mathcal{A}}(t_k)$. 
For example, a mark at time $t_k$ of type $(x,x+1)$ affected $\hat{\mathcal{A}}$ iff. $\hat{\mathcal{A}}(t_k)\cap\{x,x+1\}\neq \emptyset$.

When a mark $t_k$ is of type $b$ (resp. $b'$) and  $2\in \hat{\mathcal{A}}(t_k)$ (resp. $N-2\in \hat{\mathcal{A}}(t_k)$), 
the process \emph{branched} and a flag is \emph{created} in $\hat{\mathcal{A}}$ (we recall that the process is constructed backwards in time).
Note that this is the only way for the cardinal of $ \hat{\mathcal{A}}$ to \emph{increase}.  
When a mark $t_k$ is of type $(\pm,1)$ (resp. $(\pm, N-1)$) and  $1\in \hat{\mathcal{A}}(t_k)$ (resp. $N-1\in \hat{\mathcal{A}}(t_k)$), 
a flag is \emph{deleted} in $\hat{\mathcal{A}}$. Another way for the cardinal of $\hat{\mathcal{A}}$ to decrease is if the process meets a mark of type $c$ or $c'$ either at site $2$ or $N-2$, while site $1$ or $N-1$ was also in $\hat{\mathcal{A}}$, in which case site $2$ (resp. $N-2$) is removed from $\hat{\mathcal{A}}$, while site $1$ (resp. $N-1$) remains. The last way for the cardinal of $ \hat{\mathcal{A}}$ to decrease is hitting time $0$ in which case all remaining flags are removed.

We will call  \emph{boundary marks} the marks that make the cardinal of $\mathcal{A}$ change. Note, in particular, that not all marks of type $c$ or $c'$ are boundary marks, only those who occur while both sites of the boundary are flagged. In section \ref{label:determination}, we explain how to recover from the process's value $\eta_x(t)$ from the process $ \hat{\mathcal{A}}$ and the marks who affected it. For now, however, we study the process  $ \hat{\mathcal{A}}$ itself.

\subsection{Markov flag processes}

Recall that $\hat{\mathcal{A}}(s)$ is defined on $[0,t]$, observed backwards in time, and is right-continuous. For simplicity, we want to observe the process $\hat{\mathcal{A}}(s)$ forward in time, so that we define $(\mathcal{A}(s))_{s\in [0,t]}$ the right-continuous version of  $\hat{\mathcal{A}}(-s)$: we set $\mathcal{A}(s)=\hat{\mathcal{A}}(-s)$ on $[0,t]$ except at the time of the marks $t_k$, at which $\mathcal{A}(t_k)=\mathcal{A}(t_k^+)$. Since the Poisson point processes forward and backward in time have the same distribution, one easily checks that up until time $t$ (forward in time) $\mathcal{A}(s)$ is a Markov process on $\mathcal{P}(\Lambda_N)$ (the set of subsets of $\Lambda_N$) with generator $\mathcal{L}^\dagger_N=\mathcal{L}^{\dagger, 1}_N+\mathcal{L}^{\dagger, 2}_N$ acting on functions $f:\mathcal{P}(\Lambda_N)\to\R$. The generator $\mathcal{L}^{\dagger, 1}_N$ represents the flag's motion, either due to the stirring dynamics or the copy mechanism,
\begin{multline*}
\mathcal{L}^{\dagger, 1}_N f(A)=N^2\sum_{x=1}^{N-2}\pa{f(A^{x,x+1})-f(A)}+cN^{2-\theta}{\bf 1}_{ \{2\in A\}}\pa{f(A\cup \{2\}\setminus\{1\})-f(A)}\\+c'N^{2-\theta}{\bf 1}_{\{N-2\in A\}}\pa{f(A\cup \{N-2\}\setminus\{N-1\})-f(A)}.
\end{multline*}
whereas $\mathcal{L}^{\dagger, 2}_N$ represents the branching (creation) and death (deletion) events at the boundary,
\begin{multline}
\label{eq:DefgenAB}
\mathcal{L}^{\dagger, 2}_N f(A)=rN^{2-\theta}{\bf 1}_{ \{1\in A\}}\pa{f(A\setminus\{1\})-f(A)}+r'N^{2-\theta}{\bf 1}_{\{N-1\in A\}}\pa{f(A\setminus\{N-1\})-f(A)}\\
+bN^{2-\theta}{\bf 1}_{ \{2\in A\}}\pa{f(A\cup\{1\})-f(A)}+b'N^{2-\theta}{\bf 1}_{\{N-2\in A\}}\pa{f(A\cup \{N-1\})-f(A)}.
\end{multline}
In the expression of $\mathcal{L}^{\dagger, 1}_N $, we denoted 
\[A^{x,x+1}=
\begin{cases}
A&\mbox{ if }x, x+1\notin A\mbox{ or }x, x+1\in A\\
A\setminus\{x\}\cup\{x+1\}&\mbox{ if }x\in A,\; x+1\notin A\\
A\setminus\{x+1\}\cup\{x\}&\mbox{ if }x+1\in A,\; x\notin A.
\end{cases}
\]
the set $A$ after exchange of sites $x$ and $x+1$.  We also define
\begin{multline*}
\widetilde{\mathcal{L}}^{\dagger, 1}_N f(A)=N^2\sum_{x=1}^{N-2}\pa{f(A^{x,x+1})-f(A)}+cN^{2-\theta}{\bf 1}_{ \{2\in A\}}\pa{f(A^{1,2})-f(A)}\\+c'N^{2-\theta}{\bf 1}_{\{N-2\in A\}}\pa{f(A^{N-1,N-2})-f(A)},
\end{multline*}
which is almost identical to $\mathcal{L}^{\dagger, 1}_N$, except that when $A$ contains either $\{1,2\}$ or  $\{N-2,N-1\}$, and a mark of type $c$ or $c'$ occurs. In this case, instead of removing the flag at site $2$ or $N-2$, $\widetilde{\mathcal{L}}^{\dagger, 1}_N$ simply switches the two flags at the boundary. 
\begin{rem}
We remark that the generator $\widetilde{\mathcal{L}}^{\dagger, 1}_N$ above does \emph{not} allow branchings, which are the only way for the flags to increase. Moreover, since in the particular case of the copies mentioned above we do not remove flags but exchange them, this generator \emph{conserves} the number of flags.
\end{rem}

In what follows, we forget the time horizon $t$, and consider two Markov processes $\mathcal{A}(s)$ (resp. $\mathcal{B}(s)$) started from a set $A\subset\Lambda_N$, and driven by the generator $\mathcal{L}^\dagger_N$ (resp. $\widetilde{\mathcal{L}}^{\dagger,1}_N$). Recall the construction using the Poisson point processes introduced above, since they are equivalent, we will alternatively use both descriptions (as Markov processes, or as functions of the Poisson point processes) for the two processes $\mathcal{A}$ and $\mathcal{B}$. Recall that we refer to any element of the processes $\mathcal{N}^b$ or $\mathcal{N}^\pm$ as \emph{boundary marks} (For simplicity, we will often ignore the boundary marks of type $c$ or $c'$, we will prove later on that they occur with very small probability). The process $\mathcal{B}$ therefore evolves as $\mathcal{A}$, except that the marks that would branch or  delete flags have no effect. Note in particular that until the first time $\tau>0$ that $\mathcal{A}$ is affected by a boundary mark, we have $\mathcal{A}(s)=\mathcal{B}(s)$ for any $s<\tau$. We denote $\bP^\dagger_A$ the joint distribution of those two processes started from the set $A$  and $\bE^\dagger_A$ the corresponding expectation.

Fix a starting set $A\subset \Lambda_N$,  and give each element of $A$ an arbitrary and unique label  $n$, $n=1, \dots, |A|$. For $1\leq n\leq |A|$, we follow the evolution of the flags in $\mathcal{A}$ and $\mathcal{B}$, and denote by $X_n(t)$ (resp. $Y_n(t)$) the position of the flag labeled $n$ in $\mathcal{A}(t)$ (resp. $\mathcal{B}(t)$). Each time a new flag appears in $\mathcal{A}$, we label it by the smallest integer not used previously. Whenever a mark $(x,x+1)$ occurs while both $x$ and $x+1$ are in $\mathcal{A}$ or $\mathcal{B}$, we switch the labels of the flags at $x$ and $x+1$. In other words,
\[\mathcal{A}(t)=\{X_n(t), \; n\in K(t)\},\quad  \mbox{ and }\quad \mathcal{B}(t)=\{Y_n(t), \; n\leq |A|\}\]
where $K(t)$ is the set of flag labels currently present at time $t$ in $\mathcal{A}$. Further note that, fixed a label $ n $,  $Y_n$ is a random walk on $\Lambda_N$ reflected at the boundary, jumping at rate $N^2$ to each neighbor, and at an extra rate $cN^{2-\theta}$ (resp. $c'N^{2-\theta}$) from $2$ to $1$ (resp. $N-2$ to $N-1$). The processes $X_n$ have the same distribution, except that they branch when they meet a branching mark (i.e. at rate  $bN^{2-\theta}$, when $ X_n=2$, and at rate $b'N^{2-\theta}$ when $X_n=N-2$), and die when they meet a mark $\pm$ (i.e. at rate  $rN^{2-\theta}$ when $X_n=1$ and at rate $r'N^{2-\theta}$ when $X_n=N-1$). They can also die when a mark of type $c$, $c'$ occurs while another flag is in the same boundary, but since we will prove the probability that this happens during $\mathcal{A}$'s lifespan is very small, we do not detail this possibility.

\medskip

We first focus on the process $\mathcal{B}$. Denote by  $\tau_n$ the first time at which the flag  labeled $n$ in $\mathcal{B}$ meets a boundary mark. Denote $\tau=\tau_1\wedge\tau_2\wedge\dots\wedge \tau_m$ the first time a flag meets a boundary mark. For $n\leq |A|$, we define
\[C=\{\mbox{At time $\tau$, there was only one flag at the left boundary}\}=\{|\mathcal{B}(\tau)\cap\{1,2\}|=1\}\]
and
\[D_\pm=\{\mbox{The mark encountered at time $\tau$ was of type $\pm$}\}.\]
We start with a technical Lemma, which is the main ingredient to prove Lemmas \ref{lem:res} and \ref{lem:bcor}.
\begin{lem}
\label{lem:probdeath}
Define $\nu_\theta=\min(\theta, 1-\theta)/4$. For any $ A\subset\{1,\dots, N/2\}$ containing site $3$, and any $1\leq n\leq |A|$
\begin{equation}
\label{eq:Oktaun}
\bP^\dagger_{A}(C\cap D_\pm\cap \{\tau=\tau_n\} )= \frac{r \rho_\pm}{r+b}\bP^\dagger_{A}(\tau=\tau_n )+|A|^2O(N^{-\nu_\theta}),
\end{equation}
where $\rho_+=\bar\rho$ and  $\rho_-=1-\bar\rho$. Furthermore, for any $A\subset \Lambda_N$
\begin{equation}
\label{eq:deathbound}
\bP^\dagger_{A}(D_+\cup D_- )\geq p +|A|^3O(N^{-\nu_\theta}),
\end{equation}
where 
\[p:=\min\left\{\frac{r}{r+b},\frac{r'}{r'+b'}\right\}>\frac{1}{2}.\] 
In the identities above, the $O(N^{-\nu_\theta})$ can be chosen uniform in the initial set $A$ satisfying the relevant conditions.
\end{lem}
Although this result involves significant technical difficulties, its content is fairly simple, so that before proving it, we briefly explain those two identities. Recall that we ultimately want to estimate the boundary densities (Lemma \ref{lem:res}) and two-points correlations (Lemma \ref{lem:bcor}), so that we need to examine the value of $\eta_t(3)$ and $\eta_t(3)\eta_t(x)$ for $x$ of order $\delta N$. For the first identity, we concentrate on sets contained in the left half of the system and containing site $3$ to avoid waiting for a long time for a flag to get close to a boundary, and to know with high probability that the next boundary mark is gonna be encountered at the left boundary. Equation \eqref{eq:Oktaun} states that for any $n$, assuming that the next flag to encounter a boundary mark is labelled $n$, the probability that this boundary mark is of type $+$ (resp. $-$)  converges to $\bar\rho r/(r+b)$ (resp. $(1-\bar\rho )r/(r+b)$). The second identity states that at the next boundary mark encountered, the probability that the affected flag branches rather than dies is less than $1/2$, which ultimately ensures that w.h.p. the branching process ultimately dies out after branching a finite number of times (cf. Corollary \ref{cor:A3}).

\begin{proof}[Proof of Lemma \ref{lem:probdeath}]
This proof being a little intricate, we split it in several steps.

\medskip

\noindent{\bf{Step 1: a crude estimate on $\tau$. }} Recall \eqref{eq:Defetheta}. We first claim
\begin{equation}
\label{eq:Etau}
\bP^\dagger_A(\tau\geq N^{-\widehat{\theta}})=O(N^{-\widehat{\theta}})
\end{equation} 
for any set $ A$ containing site $3$: Fix such a set, and without loss of generality label $1$  the flag at site $3$. The proof of this first claim follows the same steps as Lemma 5.3 of \cite{ELX18}, which we adapt to the slowed-down rates at the boundary. To prove \eqref{eq:Etau}, first note that 
\[\bE^\dagger_A(\tau)\leq \bE^\dagger_A(\tau_1)= \bE^\dagger_{\{3\}}(\tau).\]
Consider therefore a single flag initially at site $3$, $Y_1(t)$ its position in $ \mathcal{B}(t)$. This flag performs excursions away from site $3$, either in the bulk $\{4,\dots,N-1\}$, or in the left boundary $\{1,2\}$, in which case it has a  probability of order $N^{-\theta}$ of encountering a boundary mark, after which it ultimately gets back to site $3$. By the Markov property, all the excursions of the flag are i.i.d. in $\mathcal{B}$.
In what follows, we will neglect the possibility that the flag meets a boundary mark at the right boundary, which would only decrease $\tau$ and is therefore not an issue.  Denote $t_i$ the time of the $i$-th visit to site $3$, i.e. $t_0=0$, and 
\[t_i=\inf\{t>t_{i-1}, \quad Y_1(t)=3, \quad  Y_1(t^-)\neq 3\}.\]
For $i>0$,  denote by $E_i$ the event 
\[E_i=\{\mbox{The flag met a left boundary mark during the $i$-th excursion $[t_{i-1},t_i )$}\}.\]
Since the excursions are i.i.d., so are the $E_i$'s and $t_i-t_{i-1}$'s, and it is elementary to show (cf.  Lemma 5.3 of \cite{ELX18}) that $\bE^\dagger_{\{3\}}(t_i-t_{i-1})=\bE^\dagger_{\{3\}}(t_1) =O(N^{-1})$ (the excursions in the boundary last a time of order $N^{-2}$, whereas those in the bulk last a time of order $N^{-1}$). We can now write 
\[\bE^\dagger_{\{3\}}(\tau)=\sum_{i>0}\bE^\dagger_{\{3\}}\pa{\tau{\bf 1}_{\{E_i\mbox{ and } E_j^c\mbox{ for }j<i\}} }\leq\sum_{i>0}i \bE^\dagger_{\{3\}}(t_1)\bP^\dagger_{\{3\}}\pa{E_i\mbox{ and } E_j^c\mbox{ for }j<i}\].
There exists a constant $c=c(r,b, c)$ such that each boundary excursion has a probability $\bP^\dagger_{\{3\}}(E_1)\geq cN^{-\theta}$ of encountering a boundary mark. This finally yields as wanted
\begin{equation}
\label{eq:tausite3}
\bE^\dagger_{\{3\}}(\tau)=O(N^{\theta-1}),
\end{equation}
the factor $ N^{\theta}$ being the expectation of the geometric number of excursions necessary before meeting a boundary mark.
This together with Markov's inequality proves \eqref{eq:Etau} uniformly for any set $A$ containing site $3$.

\medskip
\noindent{\bf{Step 2: inserting independent excursions at the left boundary.}}
Because of the subtle correlations between the flag's motions and the marks they encounter, we detail this step, which is one of the main novelties w.r.t. \cite{ELX18}. We will call \emph{excursion at the left boundary} a pair $(\tau_Z, Z(s), 0 \leq s\leq \tau_Z),$ satisfying 
\begin{enumerate}
\item $\tau_Z>0$, $Z_0=2$, for any $s<\tau_Z$, $Z(s)\in \{1,2\}$ and $Z_{\tau_Z}=3$.
\item $(Z(s))_{0 \leq s\leq \tau_Z}$ is right-continuous.
\end{enumerate}
In order not to burden the notations, when referring to an excursion, we will simply denote it $Z$, the stopping time $\tau_Z$ will be implicitly associated with it. 

\medskip

Fix a realization of the marks $ \boldsymbol{\mathcal{M}} $ on $\Lambda_N$, on which we  build $\mathcal{A}$ and $\mathcal{B}$. On the same probability space, fix an i.i.d. family  of boundary excursions $(Z^j)_{j\in \Z}$, each distributed as 
\[\bP(Z^j\in \cdot)=\bP^\dagger_{\{2\}}\pa{(Y_1(s))_{0\leq s\leq H_3}\in \cdot},\] 
where $H_3=\inf\{s\geq 0, Y_1(s)=3\}$.  To each excursion $Z^j$, we also associate boundary marks distributed according to independent Poisson point processes $\mathcal{M}^j$ on $[0, \tau_{Z^j}]$ with the same intensity as the ones appearing in $\bP^\dagger$. We take the family of excursions and their associated boundary marks independent of  $ \boldsymbol{\mathcal{M}} $. We are going to build a second process $\mathcal{B}'$, distributed as $ \mathcal{B}$, using both the marks in $ \boldsymbol{\mathcal{M}} $ and these independent excursions $Z^j$. Fix a set $A\subset\{1,\dots,N/2\}$ containing site $3$, and  set $\mathcal{B}'(0)=A$. 
The two processes $\mathcal{B}$, $\mathcal{B}'$ then follow the marks in $ \boldsymbol{\mathcal{M}} $ (recall that the boundary marks do not affect the evolution of either $\mathcal{B}$ or $\mathcal{B}'$), until the first time\gab{,} $t_1$\gab{,} that a flag is at site $2$ while no other flag is at a distance less than $\ell_N:=(\log N)^4$ from the left boundary, 
\[t_1=\inf\{ s\geq 0,  \;\mathcal{B}'(s)\cap\{1,\dots,\ell_N\}=\{2\}\}.\]
After time $t_1$, we let all the flags in $\mathcal{B}'$ evolve according to the marks, except the flag that was at site $2$ at time $t_1$, which then performs the excursion $Z^1$ until time
\[s_1=\inf\{ s\geq t_1, \; 3\in\mathcal{B}'(s)\}.\]
Note that\gab{,} most likely, we have $ s_1=t_1+\tau_{Z^1}$, because otherwise another flag has reached site $3$ before the excursion $Z^1$ finished. Since by construction all other flags were, when the excursion started, at a distance $\ell_N$ of the left boundary, this is very unlikely. If 
$s_1\neq t_1+\tau_{Z^1}$, we say that the construction \emph{failed} after inserting the excursion $Z^1$. Regardless of whether or not the construction failed,  we let $\mathcal{B}'$ evolve after time $s_1$ by following the marks in $ \boldsymbol{\mathcal{M}} $ until the second time 
\[t_2=\inf\{ s\geq s_1,  \;\mathcal{B}'(s)\cap\{1,\dots,\ell_N\}=\{2\}\}.\]
We then replace the trajectory of the flag at site $2$ by the excursion $Z^2$,  until time  
\[s_2=\inf\{ s\geq t_2, \; 3\in\mathcal{B}'(s)\}.\]
Once again, if $s_2\neq t_2+\tau_{Z^2}$ we say that the construction failed after inserting excursion $Z^2$, and then repeat the construction. More precisely, assume that $\mathcal{B}'$ has been built up until time $s_{j-1}$, it then follows the marks in $ \boldsymbol{\mathcal{M}} $ until
\[t_j=\inf\{ s\geq s_{j-1},  \;\mathcal{B}'(s)\cap\{1,\dots,\ell_N\}=\{2\}\}.\]
After $t_j$, $\mathcal{B}'(s)$ follows the marks in $ \boldsymbol{\mathcal{M}} $, except the flag present at site $2$ at time $t_j$, which follows the trajectory $Z^j$ until time $s_j=\inf\{ s\geq t_{j}, \; 3\in\mathcal{B}'(s)\}.$
If $s_j\neq t_j+\tau_{Z^j}$, we say that the construction fails after inserting $Z^j$, and we then carry on with the same scheme. Recall that we built both $ \mathcal{B}$ and $\mathcal{B}'$ on the same probability space, we denote $\bQ_A$ their joint distribution starting from the set $A$. Further note that by  Markov property,
$ \mathcal{B} \overset{(d)}{=}\mathcal{B}'.$
We denote  by $\tau'$ the first time a flag in  $\mathcal{B}'$ meets a boundary mark.

\medskip

\noindent{\bf{Step 3: estimation of the probability that the construction failed before time $\tau'$.}}
Recall that the construction fails if, after inserting an excursion $Z_j$, a flag initially at a site $x>\ell_N$ gets to site $3$ before the excursion ends.
This means either that the boundary excursions lasted more than $N^{-2}\ell_N$, or that another flag traveled a distance of order $\ell_N$ in a time $N^{-2}\ell_N$. Both of those probabilities are at most of order $|A|O(e^{-\sqrt{\ell_N}})=O(N^{1-\log N})$, i.e. for any integer $j$,
\begin{equation}
\bQ_A(\mbox{The construction failed after inserting } Z^j )=O(N^{-\log N/2})
\end{equation}
We now only need a rough bound on the number of excursions inserted before time $\tau'$. Define a discrete time random walk $M_j$  as 
\[M_{j+1}=M_j+N^{-3}{\bf 1}_{\{s_j-t_j\geq N^{-3}\}},\]
which jumps at a distance $N^{-3}$ to the right whenever an excursion is inserted that lasted longer than $N^{-3}$. Clearly, $\bP(M_{j+1}=M_j)= O(N^{-1})$ for all steps $j$ and independently from the other steps, because the jumps in $\mathcal{B}$ occur at rate $O(N^{-2})$.  In particular, by a standard large deviation estimate, $\bP(M_{N^4}\leq 1)=O(e^{-N})$. Furthermore, by construction, $M_j\geq 1\Rightarrow s_j\geq 1$. These two remarks, together with \eqref{eq:Etau}, yield that 
\[\bQ_A( s_{N^4}\leq \tau' )\leq \bQ_A( \tau'\geq 1 )+ \bQ_A( s_{N^4}\leq 1 )=O(N^{-\widehat{\theta}})+O(e^{-N})=O(N^{-\widehat{\theta}}),\]
so that putting all those bounds together yield, by union bound,
\begin{equation}
\label{eq:fail}
\bQ_A(\mbox{The construction failed before time  }\tau'  )=O(N^{-\widehat{\theta}})+O(N^4 N^{-\log N/2})=O(N^{-\widehat{\theta}}).
\end{equation}

\medskip

\noindent{\bf{Step 4: estimation of the time spent  with at least two flags close to the boundary.}}
We introduce the time sets
\begin{equation*}
\label{eq:DefIt}
I_t:=\left\{s\leq t, \mathcal{B}'(s)\cap\{1,2\}\neq \emptyset \right\}:=I^1_t\cup I_t^2\cup I_t^3,
\end{equation*}
where
\[I_t^1=\bigcup_{j\;, \; s_j\leq t}[t_j, s_j)\]
\[I_t^2=\left\{s\leq t, \mathcal{B}'(s)\cap\{1,2\}\neq \emptyset\; \mbox{ and }\;|\mathcal{B}'(s)\cap\{1,\dots, \ell_N\}|\geq 2 \right\}\]
and 
\[I^3_t=I_t\setminus( I_t^1\cup I_t^2).\]
Note that for any $t>0$, $I_t$, $I_t^1$ and $I_t^2$ can all be split into a disjoint union of a finite number of time segments of the form $[s,s')$, therefore so can $I_t^3$. We therefore write $I_3=\sqcup_j [a_j, b_j)$. Since $I_t^3\cap I_t^2=\emptyset$, during each of the segments $[a_j,  b_j)$, there is exactly one flag labeled $n=n(j)$ in $\{1,\dots, \ell_N\}$ and this flag is in $\{1,2\}$. Further note that we cannot have $Y_{n}(a_j^-)=3$, because else, since there is no other flag in $\{1,\dots,\ell_N\}$,  $a_j$ would have been the start of an excursion $Z_{j'}$, which is impossible since $I_t^3\cap I_t^1=\emptyset$. This means that at time $a_j^-$, another flag was at site $\ell_N$ and jumped to site $\ell_N+1$ at time $a_j$. On the other hand, one can check that 
\[b_j=\inf\left\{s\geq a_j, \; 2\in \mathcal{B}(s) \mbox{ or }\ell_N \in \mathcal{B}(s)\right\}.\]

\medskip

We claim that, letting
\[I_t^4=\left\{s\leq t, \;|\mathcal{B}'(s)\cap\{1,\dots, 2\ell_N\}|\geq 2 \right\},\]
we have 
\begin{equation}
\label{eq:inclusion1}
\bQ_A(I_{\tau'}^2\cup I_{\tau'}^3\not \subset I_{\tau'}^4)=O(N^{-\log N/4}).
\end{equation}
To prove this identity, first note that $I_{\tau'}^2\subset I_{\tau'}^4$, therefore we only need to prove that with high probability $I_{\tau'}^3\not \subset I_{\tau'}^4$. We already pointed out that if $[a_j,  b_j)$ is one of $I_3$'s segments, at time $a_j$, a flag jumped from $\ell_N$ to $\ell_N+1$, while another was at the left boundary. Since the segment ends whenever the boundary flag reaches site  $3$, in order to have $[a_j,  b_j)\not \subset I_1^4$, the other flag must have reached site $2\ell_N+1$ before time $b_j$. Once again, this means either that the boundary excursions lasted more than $N^{-2}\ell_N$, or that the other flag traveled a distance of order $\ell_N$ in a time $N^{-2}\ell_N$. Both of those probabilities are of order $O(e^{-\sqrt{\ell_N}})=O(N^{-\log N})$, i.e. for any integer $j$,
\begin{equation}
\label{eq:inclusionI4}
\bQ_A([a_j,  b_j)\not \subset I_{b_j}^4)=O(N^{-\log N}).
\end{equation}
We now obtain a very rough estimate of the number of segments in $I_{{\tau'}}^2$, which we bound by the total number of visits $K$ to site $\ell_N$ by any flag occurring before time $\tau'$, 
\[K=|\{t\leq \tau', \ell_N\notin \mathcal{B}'(t^-) \mbox{ and }\ell_N\in \mathcal{B}'(t)\}|.\] 
Not to burden the notation, simply denote $\bE_A $ the expectation w.r.t. $\bQ_A$. We first write $\bE_{A}(K)=\sum_{x\in A} \bE_{{\{x\}}}(K)$. Each term of the sum is less than $\bE_{{\{\ell_N\}}}(K)$. Consider therefore a single flag initially at site $\ell_N$. Each time it hits site $\ell_N$, it has a probability $1/2\ell_N$ of reaching the left boundary before getting back to site $\ell_N$. Once at the left boundary, it has a probability of order $O(N^{-\theta})$ of encountering a boundary mark. In particular, we have $\bE_{\{\ell_N\}}(K)=O(\ell_N N^{\theta})$, so that
\[\bE_{A}(|\{j\in \bN, \;b_j\leq {\tau'}\}|)\leq\bE_{A}(K)=|A|O(\ell_N N^{\theta})=O(N^2).\]
Together with \eqref{eq:inclusionI4} and Markov's inequality, this bound proves \eqref{eq:inclusion1}. 

\medskip

We now estimate $|I^4_{\tau'}|:=\int_0^{\tau'} {\bf 1}_{|\mathcal{B}'(s)\cap\{1,\dots, 2\ell_N\}|\geq 2}ds$. First, we write 
\[|I^4_{\tau'}|\leq \sum_{n\neq n'\leq |A|}\int_0^{\tau'} {\bf 1}_{\{Y_n(s), Y_{n'}(s)\in\{1,\dots, 2\ell_N\}\}}ds.\]
In particular, according to \eqref{eq:Etau} and Markov's inequality
\begin{multline*}
\bQ_A\pa{|I_{\tau'}^4|\geq N^{\theta/2-2}}=\bQ_A({\tau'} \geq N^{-\widehat{\theta}} )+\bQ_A({\tau'} \leq N^{-\widehat{\theta}} \; \mbox{ and }\; |I_{\tau'}^4|\geq N^{\theta/2-2})\\
\leq O(N^{-\widehat{\theta}})+|A|^2N^{2-\theta/2}\max_{x,y\in A}\bE_{\bQ_{\{x,y\}}}\pa{\int_0^{N^{-\widehat{\theta}}} {\bf 1}_{\{Y_1(s), Y_2(s)\in\{1,\dots, 2\ell_N\}\}}ds}
\end{multline*}
Assume that 
\begin{equation}
\label{eq:timeexc} 
\bE_{\bQ_{\{x,y\}}}\pa{\int_0^{N^{-\widehat{\theta}}} {\bf 1}_{\{Y_1(s), Y_2(s)\in\{1,\dots, 2\ell_N\}\}}ds}=O(N^{-2}\ell_N^2\log N)
\end{equation}
uniformly in $x,y\in \Lambda_N$.
Then, $\bQ_A\pa{|I_{\tau'}^4|\geq N^{\theta/2-2}}=|A|^2O(N^{-\nu_\theta})$, where $\nu_\theta$ was defined in the statement of the Lemma.
This bound together with \eqref{eq:inclusionI4} yields  
\begin{equation*}
\bQ_A(|I_{\tau'}^2\cup I_{\tau'}^3|\geq N^{\theta/2-2})=|A|^2O(N^{-\nu_\theta}),
\end{equation*}
and in particular, using \eqref{eq:Etau}, we finally obtain
\begin{equation}
\label{eq:excesstau}
\bQ_A(|I_{N^{-\widehat{\theta}}}^2\cup I_{N^{-\widehat{\theta}}}^3|\geq N^{\theta/2-2})=|A|^2O(N^{-\nu_\theta}),
\end{equation}
where as before $|I_t^2\cup I_t^3|=\int_{I_t^2\cup I_t^3}ds.$

We now only need to prove \eqref{eq:timeexc}. Since it is  quite burdensome in terms of notations, we will simply sketch the proof, it is rather elementary.  See ${\bf Y}(s)=( Y_1, Y_2)(s)$ as a two dimensional random walk on $\{(x,y)\in \Lambda_N^2,\; x\neq y\},$ reflected at the boundaries $\Delta_1=\{(1,y)\mbox{ or }(y,1) \; y\in \Lambda_N\setminus\{1\} \}$ and $\Delta_2=\{(N-1,y)\mbox{ or }(y,N-1), \; y\in \Lambda_N\setminus\{N-1\}\}$. Since we want an upper bound, we assume without loss of generality that $x,y\leq 2\ell_N$. Together, the four following claims, which we will not prove because they are elementary, prove \eqref{eq:timeexc}:
\begin{enumerate}
\item The random variable inside the expectation is bounded by $N^{-\widehat{\theta}}\leq 1$. Furthermore, since $x,y\leq 2\ell_N$, in a time $N^{-\widehat{\theta}}$ with probability $1-O(e^{-C N^{\widehat{\theta}}})$, the random walk ${\bf Y}$ never hits the boundary $\Delta_2$.
\item The random walk ${\bf Y}$ performs excursions, either in the set $\Lambda=\{1,\dots, 10\ell_N\}^2$ or in $\Lambda^c$. Each excursion in  $\Lambda$ lasts on average a time  $O(\ell_N^2N^{-2})$.
\item Each time an excursion in $\Lambda$ ends, the ${\bf Y}$ has a probability of order $\geq C/\log N$ of reaching $\Delta_2$ before hitting $\{1,\dots, 2\ell_N\}^2$, so that on average, ${\bf Y}$ performs $O(\log N)$ excursions in $\Lambda$ before hitting $\Delta_2$.
\item This means that ${\bf Y}$ spends a time $O(\log N \ell_N^2N^{-2})$ in $\{1,\dots, 2\ell_N\}^2$  before hitting $\Delta_2$. Together with the first claim, this proves \eqref{eq:timeexc}.
\end{enumerate}

\medskip

\noindent{\bf{Step 5: Proof of \eqref{eq:Oktaun}.}} We now have all the ingredients to prove \eqref{eq:Oktaun}.
Give the flags an arbitrary label at time $0$ (identical in $\mathcal{B}$ and $\mathcal{B}'$), and recall that we want to estimate  $\bP^\dagger_{A}(C\cap D_\pm\cap \{\tau=\tau_n\} )$ for any $n\leq |A|$. 
Denote by $\tau_n'$ the first time the  flag labeled $n$ in  $\mathcal{B}'$ meets a boundary mark. Finally, similarly to $D_\pm$, denote $C'$ and $D_\pm'$ the events $C'=\{|\mathcal{B}'(\tau')\cap\{1,2\}|=1\}$ and
$D'_\pm=\{\mbox{the mark met by $\mathcal{B}'$ a time $\tau'$ was of type } \pm\}$. Since both processes $\mathcal{B}$ and $\mathcal{B}'$ have the same distribution,  
\[\bP^\dagger_{A}(C\cap D_\pm\cap \{\tau=\tau_n\} )=\bQ_{A}(C'\cap D_\pm'\cap \{\tau'=\tau'_n\} ).\]
Denote by $j^*$ the index of the first excursion  $Z^j$ to meet a boundary mark in $\mathcal{M}^j,$ and let $D_\pm^*$ be the event $D_\pm^*=\{\mbox{The boundary mark met by $Z^{j^*}$ was of type $\pm$}\}$. Further denote $n^*$ the label of the flag performing the excursion $Z^{j^*}$. Recall that the flag labels go from $1$ to $|A|$, we denote $n^*=0$  if the construction failed before one of the excursions met a boundary marks. We claim that 
\begin{equation}
\label{eq:indepn*}
\bQ_{A}(C'\cap D'\cap \{\tau'=\tau'_n\} )=\bQ_{A}(D_\pm^*\cap \{n^*= n\} )+|A|^2O(N^{-\nu_\theta}).
\end{equation}
Denote $\tau_f$ the first time the construction fails, we proved in \eqref{eq:fail} that $\bQ_{A}(\tau_f\leq \tau')=O(N^{-\widehat{\theta}}).$ We now prove \eqref{eq:indepn*}. We discard the possibility that the boundary mark was encountered at the right boundary, since with probability $O(N^{-\widehat{\theta}})$, as a consequence of \eqref{eq:Etau}, no flag made it past $N/2$ before time $\tau$. The boundary mark encountered at time $\tau$ must therefore have appeared during $I_t$ defined in \eqref{eq:DefIt}. Furthermore, according to \eqref{eq:excesstau}, and since boundary marks appear at rate $O(N^{2-\theta})$, the probability that a boundary mark appeared in $I_{N^{-\widehat{\theta}}}^2\cup I_{N^{-\widehat{\theta}}}^3$ is $|A|^2O(N^{-\theta/4})$. The first  boundary mark encountered by $\mathcal{B}'$ must therefore have appeared with probability $1-|A|^2O(N^{-\nu_\theta})$ in $I_t^1$, i.e. during one of the inserted excursions, in which case, since we assumed the construction did not fail before it appeared, the mark was met by $Z^{j^*}$. This proves \eqref{eq:indepn*}, because during the inserted excursions, the flag performing the excursion is alone at the boundary.

Note that all the excursions $Z_j$ are independent from the process $(\mathcal{B}'(s))_{s\leq t_j}$ so that in particular, $n^*$ being measurable w.r.t.  $(\mathcal{B}'(s))_{s\leq t_{j^*}}$, we have 
\begin{equation}
\label{eq:tauandn}
\bQ_{A}(D_\pm^*\cap \{n^*= n\} )=\bQ_{A}(D_\pm^*)\bQ_{A}(n^*= n ).
\end{equation}
Applying the same arguments as before, one obtains straightforwardly
\begin{equation}
\label{eq:tauandn2}
\bQ_{A}(\tau'=\tau'_n )=\bQ_{A}(n^*= n )+|A|^2O(N^{-\nu_\theta}).
\end{equation}

Finally, given a flag at site $2$ performing an excursion at the boundary conditioned to meeting a boundary mark (i.e. conditioned to not jumping to site $3$ before meeting a boundary mark), letting  $p_{1,\pm}$ (resp. $p_2$, $p_3$) the probability that the flag encountered a mark $\pm$  (resp. a branching mark, resp. no boundary mark) before coming back to site $2$, one has the explicit formulas
\begin{multline*}
p_{1,\pm}=\frac{(c+N^{\theta})r\rho_\pm }{(b+c+N^{\theta})(r+N^\theta)}, \quad p_2=\frac{b}{b+c+N^{\theta}}\; \; \; \mbox{and}\;\; \; p_3=\frac{(c+N^{\theta})N^\theta}{(b+c+N^{\theta})(r+N^\theta)},
\end{multline*}
where $\rho_+=\bar\rho$ and $\rho_-=1-\bar\rho$. In particular, assuming that an excursion met a boundary mark, the probability that the first boundary mark encountered was of type $\pm$, resp. branching, is 
\begin{equation}
\label{outcome}
\begin{split}
p^N_\pm=\frac{p_{1,\pm}}{p_{1,+}+p_{1,-}+p_2}&=\frac{(c+N^{\theta})r\rho_\pm }{(c+N^{\theta})r+b(r+N^\theta)}=\frac{r\rho_\pm}{r+b}+O(N^{-\theta}) \\
 \mbox{resp. } \quad  p^N_b&=1-p^N_+-p^N_-=\frac{b}{r+b}+O(N^{-\theta}).
\end{split}
\end{equation}
In particular $\bQ_{A}(D_\pm^*)=p^N_\pm, $ which, together with \eqref{eq:indepn*}, \eqref{eq:tauandn} and \eqref{eq:tauandn2}, finally allows us to write
\begin{equation*}
\bQ_{A}(C'\cap D_\pm'\cap \{\tau'=\tau'_n\} )=\frac{r\rho_\pm}{r+b}\bQ_{A}(\tau'=\tau'_n )+|A|^2O(N^{-\nu_\theta}).
\end{equation*}
Since $\mathcal{B}'$ and $\mathcal{B}$ have the same distribution, this proves \eqref{eq:Oktaun}.

\medskip

\noindent{\bf{Step 6: Taking into account the right boundary.}}
We now prove \eqref{eq:deathbound}. By the Markov property, we first assume without loss of generality that $A\cap\{1,2, N-2,N-1 \}\neq \emptyset$, so that by an elementary adaptation of \eqref{eq:Etau}, $ \bE^\dagger_A(\tau)\leq CN^{\theta-1}$. Assuming $\tau\leq N^{-\widehat{\theta}}$,  we can remove from $A$ the flags initially at a distance more than $N^{1-\widehat{\theta}/4}$ of the boundaries, since one of those reaches the boundary before time $\tau$ with probability exponentially small in $N^{\widehat{\theta}/2}$. More precisely, let us denote 
\[A_1=A\cap \{1,\dots, N^{1-\widehat{\theta}/4}\}\mbox{ and }A_2=A\cap \{ N-N^{1-\widehat{\theta}/4}, \dots, N-1\},\]
according to \eqref{eq:Etau}, and shortening $D=D_+\cup D_-$, we have 
\[\bP^\dagger_{A}(D)=\bP^\dagger_{A_1\cup A_2}(D )+O(N^{-\widehat{\theta}}).\]
Label initially the flags in increasing order from left to right. Then, according to \eqref{eq:Oktaun} and its counterpart at the right boundary,
\begin{align*}\bP^\dagger_{A_1\cup A_2}(D ) =&\sum_{n=1}^{|A_1\cup A_2|}\bP^\dagger_{A_1\cup A_2}(D \mbox{ and } \tau=\tau_n)\\
=&\frac{r}{r+b}\sum_{n=1}^{|A_1|}\bP^\dagger_{A_1\cup A_2}(\tau=\tau_n) +\frac{r'}{r'+b'}\sum_{n=|A_1|+1}^{|A_1\cup A_2|}\bP^\dagger_{A_1\cup A_2}(\tau=\tau_n)\\
&+|A|^2O(N^{-\nu_\theta})
\;  \geq \;  p+|A|^2O(N^{-\nu_\theta})
\end{align*}
where $p$ was defined after \eqref{eq:deathbound}. This concludes the proof.
\end{proof}

Now that this Lemma is proved, we get back to the process $\mathcal{A}$ on which boundary marks have an effect. Start from $A=\{3\}$, and give the label $1$ to the flag at site $3$. Each time a flag is created by the generator \eqref{eq:DefgenAB}, it is given the smallest label not already used up until this point. For any $ t\geq 0$, we denote
\begin{equation}
\label{eq:kappat}
\kappa(t)=1+\abs{\left\{s\leq t, |{\mathcal{A}}(s)|= |{\mathcal{A}}(s^-)|+1  \right\}}.
\end{equation}
the total number of labels used up to time $s$. Finally, we define 
\begin{equation}
\label{eq:DefT}
T=\inf\{s\geq 0, \; \mathcal{A}(s)=\emptyset\}.
\end{equation}

\begin{cor}
\label{cor:A3}
Recall \eqref{eq:Defetheta}. For any $c>0$, there exists $\varepsilon_1(c)>0$ such that 
\begin{equation}
\label{eq:A31}
\bP^\dagger_{\{3\}}(\kappa(T)>c\log N)= O(N^{-\varepsilon_1(c)}),
\end{equation}
There also exists $\varepsilon_2>0$ such that  
\begin{equation}
\label{eq:A32}\bP^\dagger_{\{3\}}(T>N^{-\widehat{\theta}/2})= O(N^{-\varepsilon_2 }),
\end{equation}
\begin{equation}
\label{eq:A33}\bP^\dagger_{\{3\}}(\exists t\in[0, T], \;  \mathcal{A}(t)\ni N/2)= O(N^{-\varepsilon_2 }),
\end{equation}
\begin{equation}
\label{eq:A34}
\bP^\dagger_{\{3\}}(\mbox{$\mathcal{A}$ encountered a boundary mark while it contained $\{1,2\}$})= O(N^{-\varepsilon_2 }).
\end{equation}
\end{cor}

\begin{proof}[Proof of Corollary \ref{cor:A3}]
Fix $c>0$, We first prove \eqref{eq:A31}, which is a consequence of equation \eqref{eq:deathbound}, which states that regardless of the initial set, the probability that the first boundary mark encountered was of type $\pm$ (and therefore deleted a flag) is $p+o_N(1)>1/2$. Consider the times $0<t_1<\dots< t_{m}=T$ at which $|\mathcal{A}|$ changes. Set $t_0=0$, and consider the discrete time random walk $Y_k=|\mathcal{A}(t_k)|\in \bN$ for $k\geq 0$. Note that $Y_k$ is not a Markov process in itself since its jump rate depends on $\mathcal{A}(t_k)$. However, assuming $\kappa(T)>k$, one easily checks that $Y_{k}>0$: else, there has been less than $k$ updates of $|\mathcal{A}|$ before it became empty, so in particular the process must have branched less than $k$ times. (Note that $Y_k=0$ iff $\mathcal{A}(t_k)=\emptyset$). Furthermore, after $j$ updates of $|\mathcal{A}|$, the latter cannot be more than $1+j$. According to \eqref{eq:deathbound}, there exists a constant $C$ such that 
\[\bP^\dagger_{\{3\}}(Y_k>0)\leq\bP \pa{1+\sum_{j=1}^k B_j>0},\]
where the $B_j$'s are independent variables taking the value $-1$ (resp. $1$) w.p. $q_j:=p+C(j+1)^2N^{-\nu_\theta}$ (resp. $1-q_j$). We refer the reader to lemma 5.2 of \cite{ELX18} for more details. In particular, by Markov inequality, for any positive $\lambda$
\[\bP^\dagger_{\{3\}}(\kappa(T)>k)\leq \bP^\dagger_{\{3\}}(Y_{k}>0) \leq \bP \pa{\exp\cro{\sum_{j=1}^{k} \lambda B_j}>e^{-\lambda}}\leq  e^{\lambda} \prod_{j=1}^{k}\bE\pa{e^{\lambda B_j}}.\]
Choose $p'=1/4+p/2\in (1/2,p)$, for $N$ large enough, and any $j\leq c\log N$, we have $q_j\geq p'$, so that for any $N$ large enough and for any $\lambda>0$
\[\bP^\dagger_{\{3\}}(\kappa(T)>c\log N)\leq  e^{\lambda}\pa{p'e^{-\lambda}+(1-p')e^{\lambda}}^{c\log N}.\]
Choose $\lambda=(p'-1/2)^2$ small enough so that $p'e^{-\lambda}+(1-p')e^{\lambda}<1$
to obtain as wanted that there exists $\varepsilon_1=\varepsilon_1(c)$ such that 
\[\bP^\dagger_{\{3\}}(\kappa(T)>c\log N)\leq O( N^{-\varepsilon_1(c)}),\]
which proves \eqref{eq:A31}.

\medskip

The second bound \eqref{eq:A32} is a direct consequence of the first. Thanks to the first bound, we first write for some $\varepsilon:=\varepsilon_1(1)>0$
\[\bP^\dagger_{\{3\}}(T>N^{-\widehat{\theta} /2})\leq \bP^\dagger_{\{3\}}(T>N^{-\widehat{\theta}/2 } \; \mbox{ and } \; \kappa(T)\leq \log N)+O( N^{-\varepsilon}).\]
We then bound from above $T$ by the sum of the lifespans $T_k$ (the difference between the time the label $k$ encounters a boundary mark of type $\pm$, and the time the label $k$ is introduced) of its flags, we obtain, since the lifespan only increases when a flag starts at site $3$ instead of sites $1,2$,
\begin{align*} 
\bP^\dagger_{\{3\}}(T>N^{-\widehat{\theta} /2} \; \mbox{ and } \; \kappa(T)\leq \log N)&\leq \sum_{n=1}^{\log N}\bP^\dagger_{\{3\}}(T>N^{-\widehat{\theta} /2} \; \mbox{ and } \; \kappa(T)=n)\\
&\leq\sum_{n=1}^{\log N} \bP^\dagger_{\{3\}}\pa{\sum_{k\leq n} T_k>N^{-\widehat{\theta} /2} \; \mbox{ and } \; \kappa(T)=n}\\
&\leq\sum_{n=1}^{\log N} \bP^\dagger_{\{3\}}\pa{\max_{k\leq n} T_k>\frac{N^{-\widehat{\theta} /2}}{n} \; \mbox{ and } \; \kappa(T)=n}\\
&\leq(\log N)^2 \bP^\dagger_{\{3\}}\pa{T_1>\frac{N^{-\widehat{\theta} /2}}{\log N} \; \mbox{ and } \; \kappa(T)\leq \log N}
\end{align*}
On the event $\kappa(T)\leq n$, the flag labeled $1$ has branched at most $n$ times. In particular, we can use \eqref{eq:tausite3}
 to obtain that 
\[\bE^\dagger_{\{3\}}\pa{T_1 {\bf 1}_{\{\kappa(T)\leq \log N\}}} \leq \log N \bE^\dagger_{\{3\}}\pa{\tau},\]
where as in \eqref{eq:tausite3}, $\tau$ is the time the flag $1$ waits before something happens to it at the left boundary. The right-hand side above, according to  \eqref{eq:tausite3}, is $O(N^{\theta-1} \log N )$, therefore by Markov inequality, 
\[\bP^\dagger_{\{3\}}(T>N^{-\widehat{\theta} /2})\leq  O(N^{-\widehat{\theta}}(\log N)^4)+ O( N^{-\varepsilon}), \]
which proves \eqref{eq:A32} by choosing $\varepsilon_2$ strictly smaller than both $\widehat{\theta}$ and $\varepsilon$.

\medskip

Identity \eqref{eq:A33} is an immediate consequence of the first two : assuming that $\kappa(T)\leq \log N$ and $T\leq N^{-\widehat{\theta} }$, by union bound, the probability that a flag reaches $x=N/2$ is less than $\log N\bP_3\pa{ \sup_{0\leq t\leq N^{-\widehat{\theta} }} X(t) \geq N/2}$, which is the probability that a random walker on $\{3,\dots, N\}$, starting at site $3$, jumping at rate $N^2$ and reflected at the boundaries, visits site $N/2$ before time $ N^{-\widehat{\theta} }$, which is of order $O(e^{-N^{\widehat{\theta}/2}})$, thus proving \eqref{eq:A33}.

\medskip

Finally, \eqref{eq:A34} is a consequence of \eqref{eq:A31} and \eqref{eq:excesstau}. The latter immediately yields, for any $A\subset \{1, \dots,N/2\}$ and such that $A\cap \{1,2\}\neq \emptyset$
\begin{equation}
\label{eq:tauA12bis}
\bP^\dagger_{A}(\{1,2\}\subset\mathcal{A}(\tau^-))=|A|^2 O(N^{-\nu_\theta}).
\end{equation}
Note we relaxed slightly the assumption on the set $A$, which can contain either site $1$ or $2$ and not just $2$. This is not an issue, since with probability $1-O(N^{-\theta})$, any flag starting from site $1$ or $2$ reaches site $3$ before any boundary mark appeared. Equation \eqref{eq:tauA12bis} is therefore a simple consequence of Markov's inequality and the fact that $\nu_\theta<\theta$. Assuming $\kappa(T)\leq \log N$, we have $|\mathcal{A}(s)|\leq 1 +\log N$ for any $s\leq T$, and at most $2\log N+2$ boundary marks were encountered by $\mathcal{A}$ before time $T$. On $\{\kappa(T)\leq \log N\}\cap \{\sup_{t\leq T}\mathcal{A}(t)<N/2\}$, we therefore use the  Markov property together with  \eqref{eq:tauA12bis} $2\log N+2$ times, to obtain that the probability in the left hand-side of \eqref{eq:A34} is less than
\[(2\log N+2)(\log N+1)^2O(N^{-\nu_\theta})+\bP^\dagger_{A}\pa{\kappa(T)\geq \log N\;\mbox{ or }\;\sup_{t\leq T}\mathcal{A}(t)\geq N/2}.\]
This proves \eqref{eq:A34} and the Lemma.
\end{proof}

\subsection{Determination tree}
\label{label:determination}
Now that we studied the process $\mathcal{A}$, we can, with it, determine the nature of $\eta_t(x)$. Recall from Section \ref{unknowns} that the process $\mathcal{A}$ observed \emph{forward in time} $0\rightarrow t$, represents the evolution of the set of unknowns $\hat{\mathcal{A}}$ $t\rightarrow0$ \emph{backward in time}. To determine the value of $\eta_t(x)$, we start the process $\mathcal{A}$ from $\mathcal{A}(0)=\{x\}$, and, following the time evolution of $\mathcal{A}(s)$, build a labeled rooted tree $\mathcal{T}_s$, for $s\in [0,t]$. An example of this construction is represented in figure \ref{fig:Determination}. Define $\mathcal{T}_0$ as the trivial one-vertex rooted tree, and label $1$ its only vertex corresponding to the label of the flag in $\mathcal{A}(0)$. Then, the tree remains unchanged until the first time $\tau$ at which the process $\mathcal{A}$ encounters a boundary mark. if the mark is of type $\pm$, we give the root a unique child, labeled $\pm$. If the mark is of branching type, we give the root two children, the first one labeled $1$ as well, and the second one labeled $2$. 
 
 \begin{figure}
\includegraphics[width=14cm]{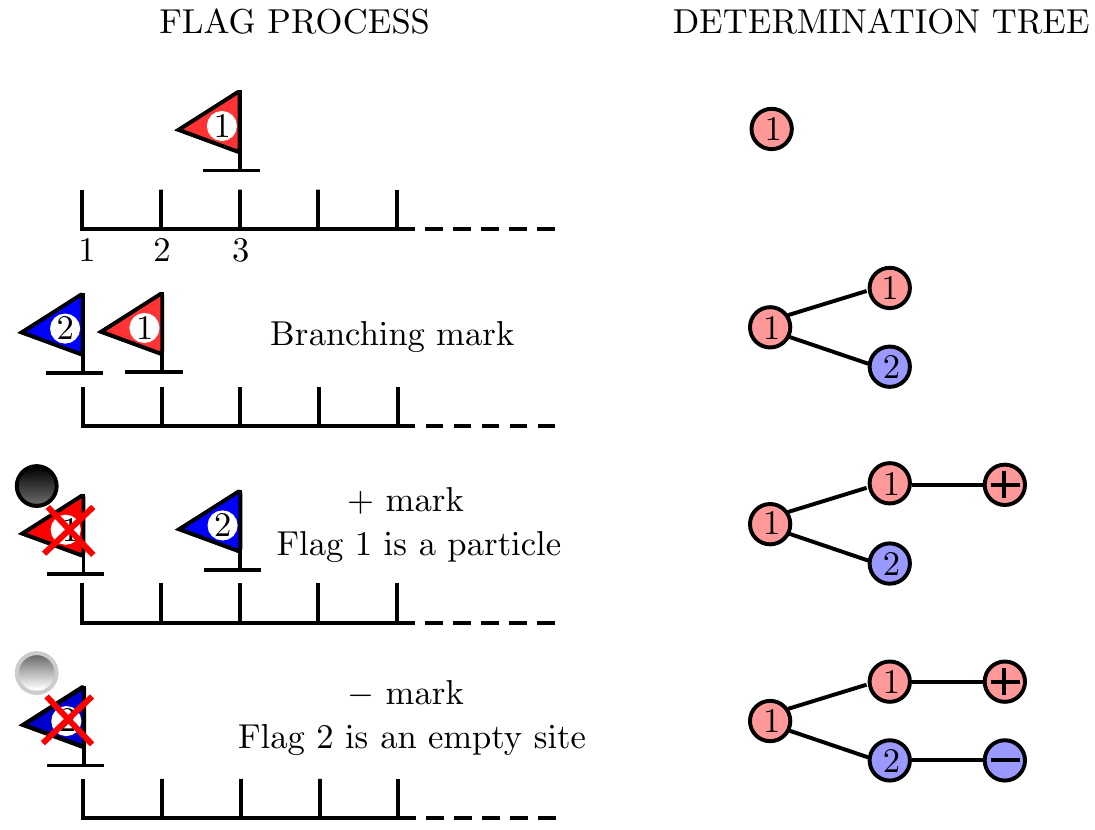}
\caption{An example of construction of the determination tree. Each time a boundary mark is encountered by one of the flags in the process $\mathcal{A}$, the tree is updated accordingly, by adding two children (for a branching mark) to each node with the affected flag's label, or one single child labelled $\pm$ (for a $\pm$ mark).}
\label{fig:Determination}
\end{figure}
 
We then carry on with the construction until time $T\wedge t$ : each time one of the flags (labeled $k$) in $\mathcal{A}$ encounters a boundary mark at time $\tau$, we build $\mathcal{T}_{\tau}$ from $\mathcal{T}_{\tau^-}$, with the following rules depending on the mark encountered by $\mathcal{A}$ at time $\tau$. Denote $k$ the label of the flag affected by the mark $\tau$.
\begin{itemize}
\item [--] If the mark is of type $\pm$, we give all leaves labeled $k$ in $\mathcal{T}_{\tau^-}$ a unique child labeled $\pm$. This is the case where the value of site $X_k(\tau)$ was updated according to a reservoir (i.e. was filled if the mark was of type $+$, or emptied if the mark was of type $-$).
\item [--] If the mark is of type $(b,2)$ (resp. $(b, N-2)$, we give all leaves labeled $k$ in $\mathcal{T}_{\tau^-}$ two children. We label the first one $k$, and the second one $k'$. The integer $k'$ is either the label of the flag at site $2$ if $2\in \mathcal{A}(\tau)$ (resp. at site $N-2$ if $N-2\in \mathcal{A}(\tau)$), or $k'=\kappa(\tau)=\kappa(\tau^-)+1$ which is the smallest unused label until now. 
This is the case where site $2$ (resp. $N-2$) was filled if site $1$ was occupied, and nothing happened otherwise.
\end{itemize}
We carry on with this construction until $T_t:=T\wedge t$. At time $T_t$, either the processes died ($\mathcal{A}(T_t)=\emptyset$), in which case the last leaf labeled $k$ received a unique child labeled $\pm$, or it didn't, and there remains in $\mathcal{T}_{T_t}$ some leaves with labels $k$ corresponding to the flag's labels in $\mathcal{A}(T_t)$. Up to this point we neglected the possibility that one of the boundary marks encountered was of type $c$ or $c'$ (recall that it requires for two flags to be at the same boundary at the time of the mark), in this case, we say that the construction failed. According to \eqref{eq:A34}, the probability that the construction of the tree fails is $O(N^{-\varepsilon_2})$.

If $T_t:=T$, we just let $\mathcal{T}_s=\mathcal{T}_T$ constant in $s\in [T,t]$. If however $T_t:= t$, in order to build $\mathcal{T}_{T_t}=\mathcal{T}_{t}$ we give each leaf labeled $k$ in $\mathcal{T}_{t^-}$ a unique child labeled $+$ (resp. $-$) if $\eta_0(X_k(t))=1$ (resp. $\eta_0(X_k(t))=1$) (where $X_k(t)$ is the position of the flag labeled $k$ in $\Lambda_N$ at time $t$). For any rooted tree $\mathcal{T}$, and any vertex $v\in \mathcal{T}$, we denote $c(v)$ the number of its children, and $ s(v)$ the number of its siblings (i.e. the number of children of its parent). We call \emph{only children} the vertices such that $ s(v)=1$.  One easily checks that, if the construction did not fail, the tree $\mathcal{T}_{t}$ has the following properties:
\begin{enumerate}
\item each leaf has a label $\pm$, and no other vertex has a label $\pm$.
\item Each vertex satisfies $c(v)\in \{0,1,2\}$. 
\item The leaves are exactly the only children, i.e. $c(v)=0$ iff $s(v)=1$.
\end{enumerate}
We denote by $\bT$ the set of rooted trees satisfying these three properties.

\medskip

 \begin{figure}
\includegraphics[width=14cm]{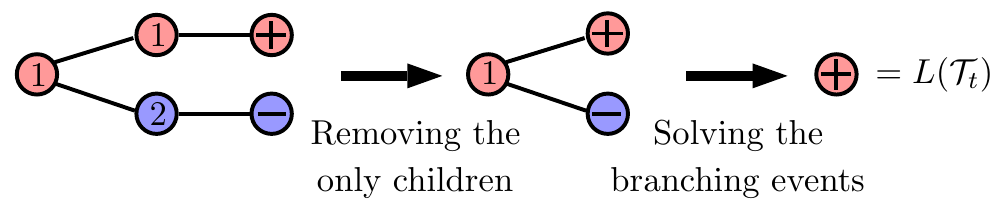}
\caption{Solving the determination tree.}
\label{fig:solve}
\end{figure}
 
Assuming the construction did not fail, we now recover, as represented in figure \ref{fig:solve} the value of $ \eta_t(x)$ by ``solving'' the tree $\mathcal{T}_{t}$. We start by deleting the \emph{only children}. For each only child deleted, we give their parent, which are now leaves, the same label $\pm$ its child had. Once this procedure is realized, there are no more only children, so that each vertex of the remaining tree has either two or zero children.
Then, one vertex at a time, we choose an arbitrary vertex $v$ with two children $(v_1, v_2)$ which are both leaves. Then, we delete $v_1$ and $v_2$ from the tree, and give $v$ : 
\begin{itemize}
\item [--] the label $+$ if $v_2$'s label is also $+$, regardless of $v_1$'s label. This is the case where there was a flag at site $1$ or $N-1$, and a mark of branching type occurred. At the time of the mark, a particle was  therefore placed at site $2$ or $N-2$.
\item [--] the same label as $v_1$ if  $v_2$'s label is $-$. This is the case where at the time of the mark of branching type, there was not a particle at site $1$ or $N-1$, so that the value of site $2$ or $N-2$ did not change.
\end{itemize}
Ultimately, this procedure deletes the entire tree except the root, and gives the root a label $ \pm:=L(\mathcal{T}_{t})$. This ``solving'' procedure can be defined for any tree in the set $\bT$, so that we see $L$ as a function $\bT \to \{+,-\}$.

\subsection{Proof of Lemmas \ref{lem:res} and \ref{lem:bcor}}
\label{sec_detfinal}
This construction is justified by the following result, which we will not prove because it is strictly analogous to Lemma 5.1, p.23 of \cite{ELX18}.
\begin{lem}
\label{lem:eta0A}
Assuming that the construction did not fail, we have 
\[\eta_t(x)=1 \quad \mbox{iff} \quad L(\mathcal{T}_{t})=+.\]
\end{lem}

Recall the definition of the outcome's probabilities of an excursion $p^N_\pm$ and $ p^N_b$ introduced in  \eqref{outcome},  we define the distribution $m^N_{\bar \rho}(\cdot)$ of a labeled random tree $\mathcal{T}\in \mathbb{T}$, built as follows :
\begin{itemize}
\item [--] The root is labeled $*$. 
\item [--] As long as there is still a leaf labeled $*$ in the tree, one such vertex chosen arbitrarily receives two children with probability $p^N_b$, each labeled $*$, and receives a unique child labeled $\pm$ w.p. $p^N_\pm$.
\item [--] The construction ends when there is no longer any leaf labeled $*$ (note that we assumed $r_2>b$, so that a.s., for $N$ large enough, this construction ends after a finite number of steps).
\end{itemize}
For any tree $\mathcal{T}\in \mathbb{T}$, we denote $|\mathcal{T}|$ its number of  vertices. The main result of this section is the following.

\begin{lem}
\label{lem:tree}
There exists $\varepsilon_3>0$ such that for any $\mathcal{T}\in \mathbb{T}$, 
\[\sup_{t\geq N^{-\widehat{\theta}}}\abs{\bP^{\dagger}_{\{3\}}(\mathcal{T}_t=\mathcal{T})-m^N_{\bar\rho}(\mathcal{T})}=O( N^{-\varepsilon_3}).\]
where the identity $\mathcal{T}_t=\mathcal{T}$ means that the structure of the tree is the same, and that the labels  $\pm$ of the leaves are identical as well.
\end{lem}

Before proving the Lemma, we state the following result.
\begin{cor}
\label{prop:valL}
There exists $\varepsilon_4>0$ such that , 
\[\sup_{t\geq N^{-\widehat{\theta}}}\abs{\rho^N_t(3)-\alpha}=\sup_{t\geq N^{-\widehat{\theta}}}\abs{\rho^N_t(3)-m^N_{\bar\rho}(L(\mathcal{T})=+)}+O(N^{-\theta})=O(N^{-\varepsilon_4}),\]
where $\alpha$ is given by \eqref{eq:alpha}.
\end{cor}

\begin{proof}[Proof of Lemma \ref{lem:tree}]
Fix the initial set $A=\{3\}.$ Given the process $\mathcal{A}$ built with the Poisson marks in $\mathcal{M}$, and given a family of independent left boundary excursions $Z^j$, we build in the same way that we built in Lemma \ref{lem:probdeath} $\mathcal{B}'$ from $\mathcal{B}$ a process $\mathcal{A}'$ by inserting the independent excursions each time a flag reaches site $2$ while no other flag is in $\{1,\dots, \ell_N\}$, where $\ell_N=(\log N)^4$. However, instead of  stopping the excursions $Z$ at the time  $\tau_{Z}$ when they reach site $3$, we stop them at the first time $\tau^m_{Z}$ they either hit site $3$ or meet a boundary mark. In particular, the time $s_j$ at which we stop the inserted excursion is 
\[s_j=(t_j+\tau^m_{Z^j})\wedge \inf\{s\geq t_j, \;3\in \mathcal{A}' (s)\}.\]

We do not detail this construction here, it is exactly identical to the one performed in step $2$ of the proof of  Lemma \ref{lem:probdeath}, except that the process is affected by the boundary marks. As for $\mathcal{B}$, the two processes $\mathcal{A}$ and $\mathcal{A}'$ have the same distribution. We also denote $\bQ_A$ their joint distribution, and $T'$ and $\kappa'(t)$ the counterparts of $T$ and $\kappa(t)$ (cf. \eqref{eq:DefT} and \eqref{eq:kappat}) for $\mathcal{A}'$. Once again, we say that this construction failed after inserting the $j$-th excursion if another flag reached site $3$ before the $j$-th excursion $Z^j$ reached site $3$ or met a boundary mark, i.e. if $s_j\neq (t_j+\tau^m_{Z^j}).$ Define the events 
\[F=\left\{\mbox{the construction failed before time }T'\right\},\]  
\[M=\left\{\mbox{A boundary mark occurred before $T'$ outside of an excursion}\right\},\]
\[O=\left\{\exists s\in [0,T'], \; \mathcal{A}'(s)\ni N/2\right\}\quad \mbox{ and } \quad S=\left\{T'\geq N^{-\widehat{\theta}}\right\}.\]
Using the same construction laid out at the beginning of Section \ref{label:determination}, we build with the process $\mathcal{A}'$ a tree $\mathcal{T}_t'$ distributed as $\mathcal{T}_t$ because $\mathcal{A}$ and $\mathcal{A}'$ have the same distribution.

\medskip

We further build a third process $\mathcal{A}''$, evolving exactly as $\mathcal{A}'$, except that:
\begin{itemize} 
\item [--] the boundary marks occurring outside of the inserted excursions are ignored.
\item [--] The time horizon $t$ is ignored as well, and when reaching time $t$, $\mathcal{A}''$ keeps evolving by following the marks in  $\mathcal{M}$ and inserting independent excursions until $T'':=\inf\{ s\geq 0, \; \mathcal{A}''(s)=\emptyset\}.$
\end{itemize}
We finally build a third tree $\mathcal{T}''_t$ using $\mathcal{A}''$ following the same construction as before. On the event $F^c\cap M^c \cap O^c\cap S^c$, we have  for any $t\geq N^{-\widehat{\theta}}$
\[\mathcal{A}''=\mathcal{A}' \mbox{ on } [0,t],\]
so that in particular $\mathcal{T}'_t=\mathcal{T}''_t=\mathcal{T}''_{T''}$. Furthermore, by construction, $\mathcal{T}''_{T''}$ is distributed according to $m^N_{\bar\rho}$. In particular, for any tree $\mathcal{T}\in\bT$, and any $t\geq N ^{-\widehat{\theta}}$
\begin{equation}
\label{eq:FMOS}
|\bQ_{\{3\}}(\mathcal{T}'_t=\mathcal{T})-m_{\bar\rho}^N(\mathcal{T})|\leq 2\bQ_{\{3\}}(F\cup M \cup O\cup S).
\end{equation}
Since $\mathcal{A} $ and $\mathcal{A'}$ have the same distribution, according to \eqref{eq:A32} and \eqref{eq:A33} in Corollary \ref{cor:A3}, 
\[\bQ_{\{3\}}(O\cup S)=O(N^{-\varepsilon_2}).\]
Furthermore, using the same arguments used to prove \eqref{eq:indepn*} and the Markov property, on the event $\kappa'(T')\leq \log N$, one straightforwardly obtains 
\begin{equation}
\label{eq:probM}
\bQ_{\{3\}}(M\cap\{\kappa'(T')\leq \log N\})=O((\log N)^3 N^{-\nu_\theta}).
\end{equation}
We do not detail this step, it is enough to use \eqref{eq:indepn*} on each interval $[\tau_i, \tau_{i+1})$ of time between two consecutive boundary marks are met. With probability $1-|\mathcal{A}(\tau_i)|^2O(N^{-\nu_\theta})$, each of those boundary marks were met during an excursion, and on $\kappa'(T')\leq \log N$ there are at most $2\log N+1$ such intervals, so that by union bound one obtains \eqref{eq:probM}. In particular, using \eqref{eq:A31} and letting $\varepsilon=\min(\varepsilon_1(1),\nu_\theta/2)$ yields $\bQ_{\{3\}}(M)=O(N^{-\varepsilon}).$

Similarly, using \eqref{eq:fail}, a union bound and Markov's property, we obtain that $\bQ_{\{3\}}(F\cap\{\kappa'(T')\leq \log N\})\leq O( N^{-\widehat{\theta}}\log N)$,  therefore  using \eqref{eq:A31} and letting $\varepsilon'=\varepsilon_1(1)\wedge (\widehat{\theta}/2)$ yields $\bQ_{\{3\}}(F)=O(N^{-\varepsilon'}).$ Finally, letting $\varepsilon_3=\min(\varepsilon_2,\varepsilon,\varepsilon')$, the right-hand side of \eqref{eq:FMOS} is $O(N^{-\varepsilon_3})$ for any $t\geq N^{\widehat{\theta}}$, which proves Lemma \ref{lem:tree}. 
\end{proof}

\begin{proof}[Proof of Corollary \ref{prop:valL}]
We start with the first identity. To prove it, one only needs to show that $m^N_{\bar\rho}(L(\mathcal{T})=+)=\alpha+O(N^-\theta)$. Define the events 
\[R_\pm=\{\mbox{the root has only one child, labeled } \pm\},\] 
and 
\[R_b=\{\mbox{the root has two children $v_1$ and $v_2$}\}.\] 
On $ R_b$, for $i=1,2$, we denote $\mathcal{T}_i$ the sub-tree of $\mathcal{T}$ composed of $v_i$ and its descendants. Note that conditionally to $R_b$,  $\mathcal{T}_1$ and  $\mathcal{T}_2$ are independent and distributed according to $m^N_{\bar\rho}$. Furthermore, by construction of the application $L$, we have the identity
\[\{L(\mathcal{T})=+\}=R_+\cup \pa{R_b\cap \{L(\mathcal{T}_2)=+\}}\cup \pa{R_b\cap \{L(\mathcal{T}_2)=-\} \cap\{L(\mathcal{T}_1)=+\}}.\]
The three events in the union above are disjoint, so that taking the measure $m^N_{\bar\rho}$ of both sides of the identity above yields, shortening $\alpha_N=m^N_{\bar\rho}(L(\mathcal{T})=+)$
\begin{equation*}
\alpha_N=p_+^N+p_b^N \alpha_N+p_b^N\pa{1-\alpha_N}\alpha_N,
\end{equation*}
which rewrites using the definition \eqref{outcome} of $p_\pm^N$ and $p_b^N$,
\[r\pa{\bar \rho-\alpha_N}+b\alpha_N\pa{1-\alpha_N}=O(N^{-\theta}),\]
which determines the boundary conditions for the equation, since it proves $\alpha_N=\alpha+O(N^{-\theta})$ defined in \eqref{eq:alpha} as wanted.

\medskip

We now prove that the supremum in the second term in the Corollary is $O(N^{-\varepsilon_4})$. Fix $\widehat{\theta}>0$ and $t\geq N^{-\widehat{\theta}}$, and shorten $\bT_k=\{\mathcal{T}\in \bT \; \mid \; |\mathcal{T}|\leq k\}$
\begin{align*}
\rho^N_t(3)&=\bP_{\{3\}}^\dagger( L(\mathcal{T}_{t})=+)\\
&= \bP^{\dagger}_{\{3\}}( L(\mathcal{T}_t)=+\mbox{ and }|\mathcal{T}_t|> k)+\sum_{\mathcal{T}\in \bT_k}\bP^{\dagger}_{\{3\}}( L(\mathcal{T})=+\mbox{ and }\mathcal{T}_t=\mathcal{T}).
\end{align*}
In particular, thanks to Lemma \ref{lem:tree}, we have for any $k$ and any $t\geq N^{-\widehat{\theta}}$
\begin{equation}
\label{eq:murhobar}
\abs{\rho^N_t(3)-m^N_{\bar\rho}( L(\mathcal{T})=+)}\leq  \bP^{\dagger}_{\{3\}}( |\mathcal{T}_t|> k) +m^N_{\bar\rho}(|\mathcal{T}_t|> k)+O(|\bT_k| N^{-\varepsilon_3}).
\end{equation}
We now obtain a crude estimate on $|\bT_k|$. Forgetting the leave's labels, to each tree $\mathcal{T}\in \bT_k$, one can associate a unique full binary tree (whose vertices all have either $0$ or $2$ children) by removing all the leaves in $ \mathcal{T}$, which are by assumption the only children. Since there are at most $k$ leaves in a tree with $k$ vertices, there are at most $2^k$ ways to associate to each leaf a label $\pm$. In particular, $|\bT_k|$ is less than $2^k$ times the number of full binary trees with less than $k$ vertices. The number of full binary trees with $k$ vertices is Catalan's number $C_{k-1}=O(4^k)$. In particular, 
\[|\bT_k|\leq 2^k \sum_{n=1}^k C_{n-1}=O(8^k).\]
We now choose $k=\frac{\varepsilon_3}{2\log 8}\log N$, which yields that the last term in \eqref{eq:murhobar} is $O(N^{\varepsilon_3/2})$. According to \eqref{eq:A31}, the first term in \eqref{eq:murhobar} is $O(N^{-\varepsilon_1\pa{\varepsilon_3/2\log 8}})$, whereas by construction of $m^N_{\bar\rho}$ the second term  is $O(N^{-\varepsilon'})$ for some positive $\varepsilon'$. Choosing 
\[\varepsilon_4=\min(\varepsilon_3/2,\varepsilon_1\pa{\varepsilon_3/2\log 8},\varepsilon')\]
 proves Corollary \ref{prop:valL}.
 \end{proof}

\begin{proof}[Proof of Lemma \ref{lem:bcor}]
In order to estimate the correlations between site $3$ and site $\ell \geq \delta N$, we start the process $\mathcal{A}$ from $\{3,\ell\}$, and denote $\mathcal{A}^{\{3\}}$, $\mathcal{A}^{\{\ell\}}\subset \mathcal{A}$ the sets of descendants of the flag initially at $3$, $\ell$. Let us denote $\tau$ the first time $\mathcal{A}^{\{3\}}$, $\mathcal{A}^{\{\ell\}}$ encounter, 
\[\tau=\sup\bra{t\leq 0, \; d\pa{\mathcal{A}^{\{3\}}(t), \mathcal{A}^{\{\ell\}}}(t)\leq 1},\]
with the convention $d(\emptyset, A)=\infty$ for any set $A$. Also denote by $T^{\{3\}}$ the lifespan of $\mathcal{A}^{\{3\}}$,
\[T^{\{3\}}=\sup\bra{t\leq 0, \; \mathcal{A}^{\{3\}}(t)=\emptyset}.\]
Then, up until time $\tau$, $\mathcal{A}^{\{3\}}$ and $\mathcal{A}^{\{\ell\}}$  can be coupled with two independent copies $\widetilde{\mathcal{A}}^{\{3\}}$ and  $\widetilde{\mathcal{A}}^{\{\ell\}}$. In particular, we can write 
\[|\varphi_t^N(3,\ell)|\leq \widehat{\bP}_{\{3,l\}}\pa{\tau> t\vee T^{\{3\}}}\leq \widehat{\bP}_{\{3,\ell\}}\pa{\tau>  T^{\{3\}}}.\]
To estimate the right-hand side, recall from \eqref{eq:A31} that with probability $1-O(N^{-\varepsilon_1(1)})$, the total number of flags $C(T^{\{3\}})$ 
created by $\mathcal{A}^{\{3\}}$ is less than $\log N$. Further recall according to \eqref{eq:A32}, $\widehat{\bP}_{\{3,k\}}(T^{\{3\}}>N^{-\widehat{\theta}})=O(N^{-\varepsilon_2})$. Finally, the probability that a flag travelled a distance at least $\delta N/4$ in a time $N^{-\widehat{\theta}}$  is of order $e^{-c\delta^2N^{\widehat{\theta}}}$. In summary, in order to have $\tau>  T^{\{3\}}$, one of three cases must have occurred.
\begin{itemize}
\item Either $\kappa(T)\geq \log N$, which occurs with probability $O(N^{-\varepsilon_1(1)})$.
\item Or $T^{\{3\}}>N^{-\widehat{\theta}}$, which occurs with probability $O(N^{-\varepsilon_2})$.
\item Or finally one of the (at most) $\log N+2$ flags (either one of the two flags initially in the system or one of the flags created) must have travelled a distance $\delta N/4$ before a time $N^{-\varepsilon}$, which by union bound occurs with probability $O(\log N e^{-c\delta^2N^{\widehat{\theta}}})$.
\end{itemize}
Letting $\varepsilon=\gab{\varepsilon_1}\wedge \varepsilon_2$, with probability $1-O(N^{-\varepsilon})$, the two processes $\mathcal{A}^{\{3\}}$, $\mathcal{A}^{\{\ell\}}$ evolve independently up until the process $\mathcal{A}^{\{3\}}$ dies. All these bounds being independent of $\ell\geq \delta N$, this proves \[\sup_{\delta N\leq \ell\leq N-1}|\varphi_t^N(3,\ell)|= O(N^{-\varepsilon}).\]
The second bound is identical.
\end{proof}

\section*{Acknowledgements}
This project has received funding from the European Research Council (ERC) under  the European Union's Horizon 2020 research and innovative programme (grant agreement   No 715734).

\end{document}